\documentclass[11 pt]{amsart}

\usepackage{amsmath, amssymb, mathtools, palatino, euler, epic, eepic, floatflt, microtype}
\usepackage[all]{xy}
\usepackage{graphicx,subcaption, epsfig}
\usepackage[usenames]{xcolor}
\usepackage{float}
\usepackage{tikz, tikz-cd, tikz-3dplot}
\usetikzlibrary{positioning}
\usetikzlibrary{shapes,arrows}
\usetikzlibrary{patterns}
\usetikzlibrary{decorations.markings}

\tikzset{->-/.style={decoration={  markings,  mark=at position #1 with {\arrow{>}}},postaction={decorate}}}

\tikzset{dot/.style={fill,circle,inner sep=1pt,minimum size=1pt}}

\usepackage{fullpage}
\usepackage{amscd,amsthm}
\usepackage{comment}
\usepackage{color}
\usepackage{newtxmath}

\usepackage[margin=1 in]{geometry}
\definecolor{darkblue}{rgb}{0,0,0.7}
\definecolor{darkred}{rgb}{0.7,0,0}
\usepackage[colorlinks,linkcolor=darkred, citecolor=darkblue,
pagebackref=true,pdftex]{hyperref}

\newcommand\defi[1]{\textit{\color{blue}#1}}    

\DeclareMathOperator\conv{conv}
\DeclareMathOperator\Hom{\overrightarrow{\text{Hom}}}
\DeclareMathOperator\oldHom{Hom}
\DeclareMathOperator\indeg{indeg}
\DeclareMathOperator\outdeg{outdeg}
\DeclareMathOperator\Poset{Poset}

\DeclareMathOperator\outN{\overrightarrow{N}}
\DeclareMathOperator\inN{\overleftarrow{N}}
\DeclareMathOperator\N{\overrightarrow{{\mathcal N}}}
\DeclareMathOperator\Nin{\overleftarrow{{\mathcal N}}}
\DeclareMathOperator\conn{conn}
\DeclareMathOperator\Aut{Aut}
\DeclareMathOperator\ch{{\chi}_o}

\newcommand \bhom {\overset{\text{\tiny$\leftrightarrow$}}{\simeq}}
\newcommand \dhom {\overset{\text{\tiny$\rightarrow$}}{\simeq}} 
\newcommand \lhom {\overset{\text{\tiny$-$}}{\simeq}}

\newcommand\R{\mathbb{R}}
\newcommand\Z{\mathbb{Z}}

\newcommand\x{\mathbf{x}}

\newcommand\diam{diam}
\newcommand\tourn{\overrightarrow{K_n}}
\newcommand\tourm{\overrightarrow{K_m}}
\newcommand\id{{\rm id}}
\newcommand\lk{{\rm lk}}

    \newtheorem{thm}{Theorem}[section]
    \newtheorem{prop}[thm]{Proposition}
    \newtheorem{lemma}[thm]{Lemma}
    \newtheorem{claim}[thm]{Claim}
    
    \newtheorem{cor}[thm]{Corollary}
    \newtheorem{defn}[thm]{Definition}
    \newtheorem{rem}[thm]{Remark}
    \newtheorem{example}[thm]{Example}
    
    \newtheorem*{theorem}{Theorem}

\begin{document}

\title[Homomorphism complexes]{Homomorphism complexes, reconfiguration, \\
and homotopy for directed graphs}

\author{Anton Dochtermann}
\address{Texas State University}
\email{dochtermann@txstate.edu}

\author{Anurag Singh}
\address{Chennai Mathematical Institute}
\email{asinghiitg@gmail.com}

\date{\today}

\maketitle

\begin{abstract}
The neighborhood complex of a graph was introduced by Lov\'asz to provide topological lower bounds on chromatic number. More general  homomorphism complexes of graphs were further studied by Babson and Kozlov.  Such `Hom complexes' are also related to mixings of graph colorings and other reconfiguration problems, as well as a notion of discrete homotopy for graphs. Here we initiate the detailed study of Hom complexes for directed graphs (digraphs). For any pair of digraphs graphs $G$ and $H$, we consider the polyhedral complex $\Hom(G,H)$ that parametrizes the directed graph homomorphisms $f: G \rightarrow H$.  This construction can be seen as a special case of the poset structure on the set of multihomomorphisms in more general categories, as introduced by Kozlov, Matsushita, and others. Hom complexes of digraphs have applications in the study of chains in graded posets and cellular resolutions of monomial ideals.

 We study examples of directed $\Hom$ complexes and relate their topological properties to certain graph operations including products, adjunctions, and foldings. We introduce a notion of a neighborhood complex for a digraph and prove that its homotopy type is recovered as the $\Hom$ complex of homomorphisms from a directed edge.  We establish a number of results regarding the topology of directed neighborhood complexes, including the dependence on directed bipartite subgraphs, a digraph version of the Mycielski construction, as well as vanishing theorems for higher homology.  The $\Hom$ complexes of digraphs provide a natural framework for reconfiguration of homomorphisms of digraphs. Inspired by notions of directed graph colorings we study the connectivity of $\Hom(G,T_n)$ for $T_n$ a tournament, obtaining a complete answer for the case of transitive $T_n$. If $G = T_m$ is also a transitive tournament, we describe a connection to mixed subdivisions of dilated simplices. Finally we use paths in the internal hom objects of digraphs to define various notions of homotopy, and discuss connections to the topology of $\Hom$ complexes. 
\end{abstract}

\maketitle

\section{Introduction}\label{sec:intro}

The study of the chromatic number of graphs and more general graph homomorphisms is an active area of research (see for instance the recent monograph \cite{HelNes}). In his seminal proof of Kneser's conjecture,    Lov\'asz \cite{Lov} introduced topological methods to the study of graph colorings via the neighborhood complex ${\mathcal N}(G)$ of a graph $G$.  Using a Borsuk-Ulam type argument, he showed that the topology (connectivity) of ${\mathcal N}(G)$ provides a lower bound on $\chi(G)$, the chromatic number of $G$.

It turns out that the neighborhood complex can be recovered (up to homotopy type) as a special case of a more general notion of a \emph{homomorphism complex} $\oldHom(T,G)$, parametrizing all graph homomorphisms from $T$ to $G$ (with ${\mathcal N}(G)$ being the special case that $T = K_2$).  Details of this construction were worked out by Babson and Kozlov in \cite{BabKozCom}. Since these original works, several authors have studied homomorphism complexes and their applications to various combinatorial problems.  This includes bounds on chromatic number coming from odd cycles and other `test graphs' (\cite{BabKozPro}, \cite{DocSch}, \cite{MatHom}, \cite{Sch09}), connections to local chromatic number (\cite{SimTar}, \cite{SimTarVre}), higher connectivity of Hom complexes (\cite{Eng},\cite{CukKozHig}), complexes of graph homomorphisms between cycles \cite{CukKozHom}, connections to Stiefel manifolds (\cite{CsoLut}, \cite{Sch}), other notions of box complexes (\cite{CsoHom}, \cite{CLSW}, \cite{MatZie}, \cite{Ziv}),  applications to hypergraphs (\cite{AloFraLov}, \cite{IriKis}, \cite{LanZie}, \cite{MatMor},  \cite{Zie}), as well as connections to statistical physics (\cite{BriWin}, \cite{DocFre}) and random graphs \cite{Kah}.  The 1-skeleton of $\oldHom(T,G)$ has also been studied in the context of discrete homotopy of graphs (\cite{DocHom}, \cite{MatBox}) and mixing of graph colorings (for the case $G=K_n$) and other reconfiguration problems (\cite{CHJconnected}, \cite{CHJmixing}, \cite{Wro}).

As Kozlov \cite{KozCom} has pointed out, there is a notion of a \emph{Hom complex} in any category where the set of morphisms between two objects can be described as a subset of some collection of `vertex set' mappings.  In this case there is notion of a \emph{multihomomorphism} between two objects, and the set of all multihomomorphisms naturally forms a partially ordered set (poset), which in turn gives rise to a topological space.  In \cite{MatMor} Matsushita considered Hom complexes of `$r$-sets', which include hypergraphs and other generalizations of graphs.  He showed how such complexes can be modeled by simplicial sets and how many results from Hom complexes of graphs, including the $\times$-homotopy theory of \cite{DocHom}, can be extended to this setting.

In this work we apply these ideas to the class of \emph{directed} graphs (or digraphs for short).  For two digraphs $G$ and $H$, we consider the complex $\Hom(G,H)$ that parametrizes all directed graph homomorphisms $f:G \rightarrow H$.  We will see that many of the categorical properties of such complexes that were satisfied in the undirected setting carry over to this context. Some of our results follow directly from the general theory of homomorphism complexes alluded to above, but in many cases we require new tools and constructions that seem specific to the directed graph setting.  At the same time we will see that the $\Hom$ complexes of digraphs exhibit behavior that is not seen in the undirected setting, providing insight into the nuances of the category of directed graphs.

We consider homomorphism complexes of digraphs to be a natural area of study in its own right, but we see how these concepts also connect to several other areas of existing research.  For instance in \cite{BraHou} Braun and Hough study a morphism complex associated to maximal chains in a graded poset.  These complexes can be recovered as $\Hom(T,G(P))$ complexes of digraphs by considering the Hasse diagram $G(P)$ of the underlying poset $P$, and choosing $T$ to be a directed path.  Hence our complex $\Hom(G,H)$ for arbitrary digraphs $G$ and $H$ generalize this construction. 

Homomorphism complexes of digraphs also make an appearance in commutative algebra.  In \cite{DocEng} the first author and Engstr\"om showed that minimal resolutions of a class of monomial ideals they call \emph{cointerval} are supported on complexes that can be described as directed homomorphism complexes.  This extends work of Nagel and Reiner \cite{NagRei} where a \emph{complex of boxes} was shown to support resolutions of (squarefree) strongly stable ideals.  The idea of using homomorphism complexes to describe resolutions of monomial ideals was further investigated by Braun, Browder, and Klee in \cite{BraBroKle}, where they considered ideals defined by nondegenerate morphisms between simplicial complexes.  It is our hope that a thorough understanding of homomorphism complexes of digraphs may lead to further applications of this kind.

In addition, the $1$-skeleton of $\Hom(G,H)$ provides a natural model to study \emph{reconfiguration} questions regarding homomorphisms of digraphs.  In many areas of graph theory, one is interested in the relationships among solutions to a given problem for a specific graph.  A predetermined rule defines how one can move from one solution to another, leading to a \emph{reconfiguration graph} of solutions. This has practical applications when exact counting of solutions is not possible in reasonable time, in which case Markov chain simulation can be used. For these questions the connectedness, diameter, realizability, and algorithmic properties of the configuration graph are typically studied (see \cite{MynNas} for a survey). In our context a pair of homomorphisms of digraphs $f,g:G \rightarrow H$ will be considered adjacent if $f$ and $g$ agree on all but one vertex, defining a reconfiguration graph that corresponds to the $1$-skeleton of $\Hom(G,H)$. For undirected graphs, the connectivity and diameter of this graph are well studied with many results and open questions, but it seems that the analogous questions for digraphs have not been explored.

Finally, the connectivity and higher topology of $\Hom(G,H)$ is also a natural place to consider a notion of \emph{homotopy} and other categorical properties for directed graphs.   For any two digraphs $G$ and $H$, the 0-cells of $\Hom(G,H)$ are given by the homomorphisms $G \rightarrow H$, and hence paths in the 1-skeleton of $\Hom(G,H)$ provide a natural notion of homotopy between them.  This perspective was investigated for undirected graphs by the first author in \cite{DocHom} and with Schultz in \cite{DocSch}, where the resulting notion was called $\times$-homotopy.  In both the undirected and directed setting, homotopy can be described by certain paths in the exponential graph $H^G$, and can also be recovered via a certain `cylinder' object.  For digraphs, there is a subtlety regarding which notion of path in $H^G$ one considers, which leads to a hierarchy of homotopies in this setting. These constructions relate to and extend existing theories from the literature (see for instance \cite{BBDL}, \cite{GLMY}).

\subsection{Our results.}
We next give a brief overview of our contributions. Our first collection of results involve structural properties of the $\Hom$ complexes that parallel those of Hom complexes in the undirected setting, and which will be used throughout the paper.  Many of these involve graph operations that induce homotopy equivalences on the relevant complexes.   Although we usually think of the $\Hom$ complexes as topological spaces (in fact polyhedral complexes), we will typically work with them as posets.  In this setting we will often describe our homotopy equivalences in terms of poset maps and a notion of homotopy equivalence of posets, which in turn induce \emph{strong homotopy equivalences} on the underlying topological spaces. We refer to Section \ref{sec:collapses} for more details.

In Proposition \ref{prop:product} we prove that the strong homotopy type of $\Hom$ complexes are preserved under products in the second coordinate.  In Propositions \ref{prop:adjoint} and \ref{prop:folding} we apply results of Matsushita from \cite{MatMor} to show that $\Hom$ complexes behave well with respect to internal hom \emph{adjunction}, and also how a notion of \emph{directed folding} preserves strong homotopy type. Our main results in the setting can be summarized as follows, where we refer to Sections \ref{sec:definitions} and \ref{sec:structural} for any undefined terms.

\begin{theorem}[Propositions \ref{prop:product}, \ref{prop:adjoint}, \ref{prop:folding}]
For digraphs $A$, $B$, and $C$, we have strong homotopy equivalences
\begin{enumerate}
    \item $\Hom(A,B \times C) \simeq \Hom(A,B) \times \Hom(A,C)$;
        \smallskip
    \item 
    $\Hom(A \times B, C) \simeq \Hom(A,C^B).$
    \end{enumerate}
If $G \rightarrow G\backslash \{v\}$ is a directed folding, then for any graph $H$ we have strong homotopy equivalences
\begin{enumerate}
    \item[(3)] $\Hom(H,G) \simeq \Hom(H,G \backslash \{v\})$; 
        \smallskip
    \item[(4)] $\Hom(G \backslash \{v\}, H) \simeq \Hom(G,H)$.
\end{enumerate}
\end{theorem}

As an important consequence of (2) we can recover $\Hom(G,H)$ (up to strong homotopy type) as the clique complex of a certain graph ${\mathcal G}(H^G)$ associated to $G$ and $H$.  See Remark \ref{rem:clique} for details.

Our next results involve notions of \emph{out-} and \emph{in-neighborhood complexes} $\N(G)$ and $ \Nin (G)$, simplicial complexes associated to a digraph $G$ that mimic the Lov\'asz construction in the undirected setting (see Definition \ref{def:neighbor}). In Theorem \ref{thm:homandnbd} we prove that both complexes are homotopy equivalent to the homomorphism complex $\Hom(K_2,G)$ (and hence to each other). In Proposition \ref{prop:universal} we prove that any simplicial complex can be recovered up to isomorphism as $\N(G)$ for some digraph $G$.  This stands in contrast to the undirected setting, where the neighborhood complex must be homotopy equivalent to a space with a free ${\mathbb Z}_2$-action.

In the undirected graph setting, the Mycielskian $\mu(G)$ of a graph $G$ has important applications to the study of homomorphism complexes.  As Csorba \cite{Cso} has shown, the neighborhood complex of $\mu(G)$ is homotopy equivalent to $S({\mathcal N}(G))$, the \emph{suspension} of ${\mathcal N}(G)$. Motivated by this, we introduce a number of directed versions of the Mycielskian (see Definition \ref{def:Mycielskian}), and in Proposition \ref{prop:Mycielskian} we determine their effect on the homotopy type of the directed neighborhood complex. In particular, we prove that for any digraph $G$ there is a homotopy equivalence
\[\N(M^3(G))  \simeq S (\N(G)),\]
\noindent
where $M^3(G)$ is a certain directed Mycielskian of $G$ and $S(-)$ denotes suspension.

We next address the effect that directed bipartite subgraphs $\overrightarrow{K}_{m,n}$ (see Definition \ref{def:bipartite}) have on the topology of directed neighborhood complexes.  In Proposition \ref{prop:bipartite} we prove that if $G$ is a digraph \emph{not} containing a copy of $\overrightarrow{K}_{m,n}$ (for any $m+n = d$) then the complex $\N(G)$ admits a strong deformation retract onto a complex of dimension at most $d-3$.  Again this mimics an analogous property for neighborhood complexes of undirected graphs, first established by Kahle in \cite{Kah}.

Our main result in this section is Theorem \ref{thm:tophomology}. Here we prove that if $G$ is any simple digraph on at most $m = 2n+2$ vertices then $\tilde{H}_i(\N(G)) = 0$ for all $i \geq n$, and that this property holds for any induced subcomplex (such a complex is said to be \emph{$n$-Leray}).  We also show that this result is tight in the sense that there exists a digraph $T_m$ on $m=2n+3$ vertices with $\N(T_m) \simeq \mathbb{S}^n$. Our results for neighborhood complexes of digraphs can be summarized as follows.

\begin{theorem}[Theorems \ref{thm:homandnbd}, \ref{thm:tophomology}, Propositions \ref{prop:universal}, \ref{prop:Mycielskian}, \ref{prop:bipartite}]
For any digraph $G$, we have homotopy equivalences
\[\N(G) \simeq \Nin(G) \simeq \Hom(K_2, G).\]
Furthermore, the neighborhood complex $\N(G)$ also satisfies the following properties.
\begin{enumerate}
    \item (Universality) Any simplicial complex $X$ can be realized as $\N(G)$ for some choice of $G$.
    \smallskip
    \item(Suspension) For any digraph $G$, we have a homotopy equivalence
    \[\N(M^3(G))  \simeq S (\N(G)).\]
    \item(Bipartite subgraphs) If a digraph $G$ does not contain a copy of $\overrightarrow{K}_{m,n}$ (for any $m+n = d$), then the complex $\N(G)$ admits a strong deformation retract onto a complex of dimension at most $d-3$.
        \smallskip
    \item(Leray property) If $G$ is a simple digraph on at most $2n+2$ vertices, then $\N(G)$ is $n$-Leray.
    \end{enumerate}
\end{theorem}

We next turn our attention to homomorphism complexes of the form $\Hom(G,T_n)$, where $T_n$ is a \emph{tourmanent} (directed complete graph) on $n$ vertices.  Homomorphisms $G \rightarrow T_n$ are used to define a notion of oriented chromatic number $\chi_o(G)$ and is a natural place to look for results that compare to Hom complexes of undirected graphs. In particular, we are interested in \emph{reconfiguration} questions of homomorphisms into tournaments as an analogue of the well-studied question of mixings of (undirected) graph colorings. In this context one is interested in the connectivity and diameter of the complex $\Hom(G,T_n)$.

We mostly study the case that $T_n = \tourn$ is an \emph{acyclic} (or \emph{transitive}) tournament.  Our main result in this section is Theorem \ref{thm:tourncollapse}, where we prove that for any digraph $G$, the complex $\Hom(G,\overrightarrow{K_n})$ is either empty or collapsible. This again stands in stark contrast to the undirected setting, where even the connectivity of $\oldHom(G,K_3)$ is a subtle question (see Section \ref{sec:reconfig} for more discussion). Having established connectivity of $\Hom(G,\tourn)$, it is then natural to consider the diameter of its $1$-skeleton.  In Theorem \ref{thm:homgconnected} we prove that $(\Hom(G,\overrightarrow{K_n}))^{(1)}$ is either empty or has diameter at most $|V(G)|$. 

In the special case that $G = \tourm$ is itself a transitive tournament we can say more about the topology and polyhedral structure of $\Hom(\tourm, \tourn)$.  In Proposition \ref{prop:mixedsub} we show that such complexes can be recovered as certain \emph{mixed subdivisions} of a dilated simplex, and are hence homeomorphic to an $(n-m)$-dimensional ball for any $m \leq n$.  Our results for complexes of homomorphisms into transitive tournaments can be summarized as follows.

\begin{theorem}[Theorems \ref{thm:tourncollapse}, \ref{thm:homgconnected}, Proposition \ref{prop:mixedsub}]
Let $\tourn$ denote the transitive tournament on $n$ vertices. Then we have

\begin{enumerate}
\item
For any digraph $G$, the complex $\Hom(G,\tourn)$ is empty or contractible.
\smallskip

    \item 
    If $\Hom(G,\tourn)$ is nonempty, then the diameter of its $1$-skeleton satisfies
\[\diam((\Hom(G,\overrightarrow{K_n}))^{(1)}) \leq |V(G)|;\]

\item
If $m \leq n$, then $\Hom(\tourm, \tourn)$ is homeomorphic to a mixed subdivision of the dilated simplex $m\Delta^{n-m}$.
\end{enumerate}
\end{theorem}

It is an open question to determine the possible homotopy types of $\Hom(G,T_n)$ for other choices of tournaments $T_n$. In Proposition \ref{prop:spheres} we show that any sphere ${\mathbb S}^n$ can be recovered up to homotopy type as $\Hom(K_2, T_n)$ for some choice of tournament $T_n$.

Our last collection of results involves applications of $\Hom$ complexes to notions of \emph{homotopy} for directed graphs. For any two digraphs $G$ and $H$, one can see that a path in $\Hom(G,H)$ corresponds to a certain `bi-directed' path in the directed graph $H^G$, where $H^G$ is the \emph{internal hom} object associated to the categorical product. This defines a notion of \emph{bihomotopy} of digraph homomorphisms, a special case of the \emph{strong homotopy} of $r$-sets as developed by Matsushita \cite{MatMor}.  Using this we show that foldings of digraphs are intimately related to bihomotopy and show that bihomotopy equivalence of digraphs is characterized by topological properties of the $\Hom$ complexes.  In Theorem \ref{thm:dismant} we see that a digraph $G$ is dismantlable if and only if $\Hom(T,G)$ is connected for all digraphs $T$, providing a directed graph analogue of a result of Brightwell and Winkler from \cite{BriWinGibbs}.

Other notions of paths in $H^G$ lead to increasingly weaker notions of homotopy for digraph homomorphisms.  The existence of a \emph{directed} path in $H^G$ defines a notion of \emph{dihomotopy} $f \dhom g$, whereas a path in the underlying undirected graph of $H^G$ defines a \emph{line homotopy} $f \simeq g$.  Dihomotopy is an example of a `directed homotopy theory', whereas line homotopy is related to constructions of Grigor’yan, Lin, Muranov, and Yau from \cite{GLMY}.  Our results can be summarized as follows.

\begin{theorem}[Propositions \ref{prop:equiv}, \ref{prop:homstrength}, Theorem \ref{thm:dismant}]
For any two digraphs $G$ and $H$, paths in the directed graph $H^G$ give rise to notions of bidirected homotopy $\bhom$, directed homotopy $\dhom$, and line homotopy $\lhom$. 
\begin{enumerate}
\item
For any pair of digraph homomorphism $f,g: G \rightarrow H$, we have a strict hierarchy
\[f \bhom g \Rightarrow f \dhom g \Rightarrow f \lhom g.\]

\item
Bihomotopy $\bhom$ leads to a notion of bihomotopy equivalence which can be characterized in terms of the topology of $\Hom$ complexes.

\item
Foldings $G \rightarrow G-v$ preserve bihomotopy type, and each bihomotopy equivalence class contains a unique stiff representative.

\item
$G \bhom {\bf 1}$ if and only if $\Hom(T,G)$ is connected for any digraph $T$.

\end{enumerate}
\end{theorem}

In the undirected setting, an important application of neighborhood and more general homomorphism complexes was the notion of `topological lower bounds' on chromatic number.  These results came from a general obstruction theory based on the ${\mathbb Z}_2$-equivariant topology of the underlying complexes.  In the directed graph setting we will see how a similar obstruction theory can be achieved, where the most natural example is given by taking the oriented cycle as $C_3$. We refer to Section \ref{sec:further} for more discussion, where we also consider some open questions.

\subsection{Paper outline}
The rest of the paper is organized as follows.  In Section \ref{sec:definitions} we discuss basic properties of the category of directed graphs and review some relevant tools from poset topology. Here we also provide the precise definition of the complex $\Hom(G,H)$ and discuss some examples.  In Section \ref{sec:structural} we establish structural results regarding the $\Hom$ complexes, including product and adjunction formulas. In Section \ref{sec:neighbor} we define the out- and in- neighborhood complexes $\N(G)$ and $\Nin(G)$ of a digraph $G$, and establish the homotopy equivalence with $\Hom(K_2,G)$.  Here we establish the other topological properties of $\N(G)$ discussed above, including universality results, the directed analogue of the Mycielski construction, the dependence on bipartite subgraphs, and vanishing of higher homology.  In Section \ref{sec:tourn} we study homomorphism complexes of the form $\Hom(G,T_n)$, where $T_n$ is a tournament. Here we prove that $\Hom(G,\tourn)$ is empty or contractible for any digraph $G$, and establish the bound on its diameter. We also discuss the connection to mixed subdivisions for the case of $G = \tourm$. In Section \ref{sec:homotopy} we define the various notions of homotopy of directed graphs and discuss how they relate to topological properties of the $\Hom$ and other complexes. Finally in Section \ref{sec:further} we consider open questions and directions for future research.

{\bf Acknowledgements.} We thank Nikola Milic\'evi\'c, Raphael Steiner, and Hiro Tanaka for insightful conversations. We are also grateful to the anonymous referees for their very helpful comments and corrections which greatly improved the paper. The first author is partially supported by Simons Foundation Grant $\#964659$. The second author is partially supported by the Start-up Research Grant SRG/2022/000314 from SERB, DST, India and by the Research Initiation Grant (RIG), IIT Bhilai, India.

\section{Definitions and objects of study}\label{sec:definitions}

\subsection{The category of directed graphs} \label{sec:graphs}

We briefly review some basics of directed graphs.  It seems these ideas are mostly well-known and scattered throughout the literature, but we collect the main ideas here.

For us a \defi{directed graph} (or \defi{digraph}) $G = (V(G), E(G))$ consists of a vertex set $V(G)$ and an edge set $E(G) \subset V(G) \times V(G)$.  Hence our digraphs have at most one (directed edge) from any vertex to another, and may have \emph{loops} of the form $(v,v)$.  Note that we do allow both $(v,w) \in E(G)$ and $(w,v) \in E(G)$. Unless otherwise specified the vertex set $V(G)$ will be finite. If $G$ is a digraph, we let $G^{{o}}$ denote the subgraph of $G$ induced on the set of looped vertices. If $(v,w)$ is an edge in $G$, we will often write $v \sim w$ and say that `$v$ is adjacent to $w$' (note that the order of $v$ and $w$ matters here). If $G$ is a digraph and $v \in V(G)$, we define the \defi{out-neighborhood} and \defi{in-neighborhood} of $v$ as
\[\outN_G(v) = \{w \in V(G): vw \in E(G)\},\]
\[\inN_G(v) = \{w \in V(G): wv \in E(G)\}.\]

For any two digraphs $G$ and $H$, a \defi{(directed) graph homomorphism} is a vertex set mapping $f: V(G) \rightarrow V(H)$ that preserves adjacency, so that if $(x,y) \in E(G)$ we have $(f(x), f(y)) \in E(H)$.  We will sometimes refer to a digraph homomorphism as a digraph \emph{map}.  We let $\Hom_0(G,H)$ denote the set of all digraph homomorphisms from $G$ to $H$.
 
If $G$ and $H$ are digraphs, the \defi{product} $G \times H$ is the directed graph with vertex set $V(G \times H) = V(G) \times V(H)$ and with adjacency given by $((g,h), (g^\prime, h^\prime)) \in E(G \times H)$ if $(g, g^\prime) \in E(G)$ and $(h, h^\prime) \in E(H)$.  This is indeed the categorical product in this category of digraphs, in the sense that given homomorphisms $f_1: T \rightarrow G$ and $f_2: T \rightarrow H$ there exists a unique map $f: T \rightarrow G \times H$ such that $f_1 = \pi_1 \circ f$ and $f_2 = \pi_2 \circ f$.
Here $\pi_1: G \times H \rightarrow G$ and $\pi_2:G \times H \rightarrow H$ are the natural projection homomorphisms.

As is the case with usual (undirected) graphs, the category of directed graphs enjoys an internal hom which is adjoint to the product.  Given digraphs $G$ and $H$ the \defi{exponential graph} $H^G$ is the directed graph with vertex set given by all vertex set mappings $f:V(G) \rightarrow V(H)$ with adjacency given by $(f,g)$ is a directed edge if whenever $(v,v^\prime) \in E(G)$ we have $(f(v), g(v^\prime)) \in E(H)$.  With this definition one can check that for any digraphs $G$, $H$, and $K$, we have a natural bijection of sets
\[\varphi: \Hom_0(G \times H, K) \xrightarrow{\cong} \Hom_0(G, K^H).\]
\noindent
Here a homomorphism $f:G \times H \rightarrow K$ is sent to $\varphi(f): G \rightarrow K^H$, where $\varphi(f)(g)(h) = f(g,h)$ for any $g \in G$ and $h \in H$.

\subsection{Collapses, posets and strong homotopy} \label{sec:collapses}

We assume the reader is familiar with basic notions of simplicial complexes and homotopy theory, but we review some concepts that we will need.  

For a simplicial complex $X$ we let $\tilde{H}_i(X)$ denote its $i$th integral reduced homology group. For a face $\sigma \in X$, the \defi{link} $\lk_X(\sigma)$ is the subcomplex defined by
\[\lk_X(\sigma) = \{\tau \in X: \tau \cap \sigma = \emptyset, \tau \cup \sigma \in X\}.\]

We will often identify a simplicial complex $X$ with its geometric realization and for instance will write $X \simeq Y$ to denote that $X$ and $Y$ are homotopy equivalent.  For any integer $n$ we let $\Delta^n$ denote the $n$-dimensional simplex.  For a finite set $S$, we let $\Delta^S$ denote the $(|S|-1)$-dimensional simplex whose vertices are the elements of $S$.

Suppose $X$ is a simplicial complex and $\sigma, \tau \in X$ are faces such that $\tau \subsetneq \sigma$ and  $\sigma$ is the only maximal face (facet) in $X$ that contains $\tau$.
  A  \defi{simplicial collapse} of $X$ is the simplicial complex $Y$ obtained from $X$ by
  removing all those simplices $\gamma$  of $X$ such that
  $\tau \subseteq \gamma \subseteq \sigma$. Here $\tau$ is called a \defi{free face} of
  $\sigma$ and $(\tau, \sigma)$ is called a \defi{collapsible pair}. We denote this collapse
  by $X \Searrow Y$. Observe that if 
  $X \Searrow Y$ then the geometric realizations of $X $ and $Y$ are homotopy equivalent.
 In this article, we write $X \Searrow  \langle A_1,A_2,\dots,A_r \rangle$ to mean that $X$ collapses onto a subcomplex generated by the faces $A_1,\dots, A_r$.

As we are mostly dealing with (realizations of) partially ordered sets (posets) and poset maps, we recall some relevant constructions specific to this context.  First note that if ${X}$ is a simplicial complex its \defi{face poset} ${\mathcal P}(X)$ is a partially ordered set given by inclusion. On the other hand, if $P = (P,\leq)$ is a poset, we use ${\mathcal O}(P)$ to denote the \defi{order complex} of $P$, by definition the simplicial complex on vertex set $P$ with $d$-dimensional simplices given by all chains $p_{i_0} < p_{i_1} < \cdots < p_{i_d}$.  Slightly abusing notation we use $|P|$ to denote the \defi{geometric realization} of the simplicial complex ${\mathcal O}(P)$.    Furthermore we will often speak about topological properties of $P$ itself, by which we mean those $|P|$. Finally note that if $X$ is a simplicial complex (or more generally a regular CW complex) then $|{\mathcal O}({\mathcal P}(X))|$ recovers the \defi{barycentric subdivision} of the complex $X$.

If $P$ and $Q$ are posets, a \defi{poset map} (or \emph{order-preserving map}) $f:P \rightarrow Q$ is a function of the underlying sets such that for all $x,y \in P$, if $x \leq_P y$ then $f(x) \leq_Q f(y)$. A poset map $f: P \rightarrow Q$ induces a simplicial map ${\mathcal O}(f): {\mathcal O}(P) \rightarrow {\mathcal O}(Q)$ of simplicial complexes and hence a continuous map $|f|: |P| \rightarrow |Q|$ on the underlying geometric realizations.  For posets $P$ and $Q$, we define the \defi{product} $P \times Q$ to be the poset with elements $\{(p,q): p \in P, q \in Q$ with relations $(p,q) \leq (p^\prime, q^\prime)$ if $p \leq_P p^\prime$ and $q \leq_Q q^\prime$.

If $P$ and $Q$ are posets, we define $\Poset(P,Q)$ to be the poset with elements given by all poset maps $f: P \rightarrow Q$, with relation $f \leq g$ if $f(x) \leq g(x)$ for every element $x \in P$.  We say that two poset maps $f,g: P \rightarrow Q$ are \defi{strongly homotopic} if they are in the same connected component of $|\Poset(P,Q)|$. If $P$ and $Q$ are strongly homotopic, then the simplicial complexes ${\mathcal O}(P)$ and ${\mathcal O}(Q)$ are strongly homotopic in the sense of \cite{BarMin}, and in particular $|P|$ and $|Q|$ are simple homotopy equivalent.

In particular, if $c:P \rightarrow P$ satisfies $c \circ c = c$ and $c(p) \geq p$ for all $p \in P$ then we say that $c$ is a \defi{closure map}.  In this case $c:P \rightarrow c(P)$ induces a strong deformation retract of the associated spaces, see \cite{Bjo95} for more discussion. We will also need the following poset fiber theorem, first established by Babson and Kozlov in \cite{BabKozCom}.

\begin{lemma}\cite[Proposition 3.2]{BabKozCom}\label{lem:posetfiber}
Let $\varphi: P \rightarrow Q$ be a map of finite posets.  If $\varphi$ satisfies
\begin{itemize}
    \item 
    $\varphi^{-1}(q)$ is contractible for all $q \in Q$, and
    \item
    for every $p \in P$ and $q \in Q$ with $\varphi(p) \geq q$ the poset $\varphi^{-1}(q) \cap P_{\leq p}$ has a maximal element,
\end{itemize}
then $\varphi$ is a homotopy equivalence.
\end{lemma}

\subsection{Discrete Morse theory}
We will need some basics from discrete Morse theory, as developed by Forman in \cite{For}.  The theory is typically described in terms of operations on a simplicial complex but can equally well be phrased in terms of underlying posets, which is the perspective we use here.

For a poset $P$ with covering relation $\prec$, a \defi{partial matching} on $P$ is by definition a subset $S \subset P$ and an injective function $\mu: S \rightarrow P \backslash S$ such that $\mu(x) \prec x$ for all $x \in S$.  The elements of $P \backslash (S \cup \mu(S))$ that are not matched are called \defi{critical}.  A matching on $S \subset P$ is \defi{acyclic} if there does not exists a sequence $x_1, \dots, x_t \subset S \cup \mu(S)$ with the property that $\mu(x_1) \prec x_2, \mu(x_2) \prec x_3, \dots, \mu(x_t) \prec x_1$. 

\begin{prop}\cite[Proposition 5.4]{For} \label{prop:dmtmain}
Let $X$ be a polyhedral complex (or any regular CW-complex) and suppose $Y \subset X$ is a subcomplex.  Then the following are equivalent.
\begin{itemize}
    \item There exists a sequence of collapses $X \Searrow Y.$
    \item There exists a partial acyclic matching on ${\mathcal P}(X)$ with critical cells given by ${\mathcal P}(Y)$.
    \end{itemize}
\end{prop}

\subsection{Homomorphism complexes}

We next turn to the main definition of the paper.   For this we follow closely the construction of the Hom complex of undirected graphs as studied in \cite{BabKozCom}. Here if $G$ and $H$ are directed graphs we define a \defi{multihomomorphism} to be a map $\alpha: V(G) \rightarrow 2^{V(H)} \backslash \{\emptyset\}$ such that if $(v,w) \in E(G)$ we have $\alpha(v) \times \alpha(w) \subset E(H)$.  We let $\Delta^{V(H)}$ denote the simplex whose vertex set is $V(H)$, and use $C(G,H)$ to denote the polyhedral complex given by the direct product $\prod_{x \in V(G)} \Delta^{V(H)}$. The cells of $C(G,H)$ are given by direct products of simplices $\prod_{x \in V(G)} \sigma_x$.

\begin{defn}
Suppose $G$ and $H$ are digraphs. Then \defi{$\Hom(G,H)$} is the polyhedral subcomplex of $C(G,H)$ with cells given by all multihomomorphisms $\alpha: V(G) \rightarrow 2^{V(H)} \backslash \{\emptyset\}$.  An element
\[\prod_{x \in V(G)} \sigma_x \in C(G,H)\]
\noindent
is in $\Hom(G,H)$ if and only if for all $(x,y) \in E(G)$ we have $(u,v) \in E(H)$ for any $u \in \sigma_x$ and $v \in \sigma_y$. In particular, the vertex set of $\Hom(G,H)$ is given by the set $\Hom_0(G,H)$ of all digraph homomorphisms $f:G \rightarrow H$.
\end{defn}

Note that for digraphs $G$ and $H$ the set of all multihomomorphisms naturally forms a poset $P(G,H)$, where $\alpha \leq \beta$ if $\alpha(v) \subset \beta(v)$ for all $v \in V(G)$.  The poset $P(G,H)$ can be seen to be the face poset of the regular CW-complex $\Hom(G,H)$.  As above we let $|P(G,H)|$ denote the geometric realization of this poset. We then have that $|P(G,H)|$ is the \emph{barycentric subdivision} of $\Hom(G,H)$, so that in particular $|P(G,H)|$ and $\Hom(G,H)$ are homeomorphic.  In many of our proofs we will think of $\Hom(G,H)$ as poset, by which we mean the poset $P(G,H)$ described above.

As we mentioned above, Kozlov has pointed out that there is a notion of a `homomorphism complex' for any category of objects where the underlying morphisms are set maps satisfying some condition.  Cells are defined as set valued functions where each restriction is a morphism in the category, see \cite{KozCom} for more details. A notion of homomorphism complex for `$r$-sets' has been studied by Matsushita in \cite{MatMor}.  We further discuss this notion in Section \ref{sec:structural}.

If $G$ and $H$ are directed graphs, we let $\overline{G}$ and $\overline{H}$ denote the underlying undirected graphs.  We note that $\Hom(G,H) \subset \oldHom(\overline{G},\overline{H})$, since an element $\alpha \in \Hom(G,H)$ is in particular a multihomomorphism of the underlying undirected graphs. We can also recover any complex $\oldHom(G,H)$  of undirected graphs via our construction in the `usual' way of embedding graphs into the category of directed graphs.   Namely, given undirected graphs $G$ and $H$ we construct digraphs $\hat{G}$ and $\hat{H}$ where for each edge in the underlying graph we introduce a  directed edge in both directions. One can then see that $\Hom(\hat{G}, \hat{H}) = \oldHom(G,H)$.  Hence the construction of $\Hom$ complexes for digraphs is in particular a generalization of the theory for graphs.

We also point that if $P$ and $Q$ are graded posets then a \emph{strictly} order preserving poset map $P \rightarrow Q$ can be thought of as a homomorphism of digraphs $G(P) \rightarrow G(Q)$, where $G(P)$ is the Hasse diagram of $P$ thought of as a digraph.  Hence our construction of $\Hom$ complexes for digraphs generalizes the work of Braun and Hough in \cite{BraHou}, where the topology of complexes of maximal chains in a graded poset is studied.

\subsubsection{Examples}

In this section we provide some examples of $\Hom$ complexes, some of which will be referred to in later sections.  Here we let $\overrightarrow{L}_n$, $\overrightarrow{C}_n$ and $\tourn$ denote the directed path graph $1\rightarrow 2\rightarrow \dots \rightarrow n$, directed cycle graph $1\rightarrow 2\rightarrow \dots \rightarrow n\rightarrow 1$, and the transitive $n$-tournament (see Section \ref{sec:tourn} for a formal definition), respectively.  We have the following easy observations. 
\begin{itemize}
\item $\Hom(\overrightarrow{L}_r, \overrightarrow{L}_s)$ is a disjoint union of $s-r+1$ points if $s\geq r$, and is empty otherwise.
\smallskip
\item $\Hom(\overrightarrow{C}_r, \overrightarrow{C}_s)$ is a disjoint union of $s$ points if $s$ divides $r$, and is empty otherwise.
\smallskip
\item $\Hom(\overrightarrow{K}_{n-1},\tourn)$ is a path on $n$ vertices.
\end{itemize}
For the next examples we define graphs $C_3^1$ and $O_6$ as follows.
\begin{itemize}
    \item Let $C_3^1$ be the digraph on vertex set $[4]$ and $E(C_3^1)=E(\overrightarrow{C}_{3})\sqcup \{(4,1)\}$.
    \smallskip

    \item
    Let $O_6$ be the digraph given by the 1-skeleton of the octahedron with orientation as in Figure \ref{fig:example5}.
    
\end{itemize}

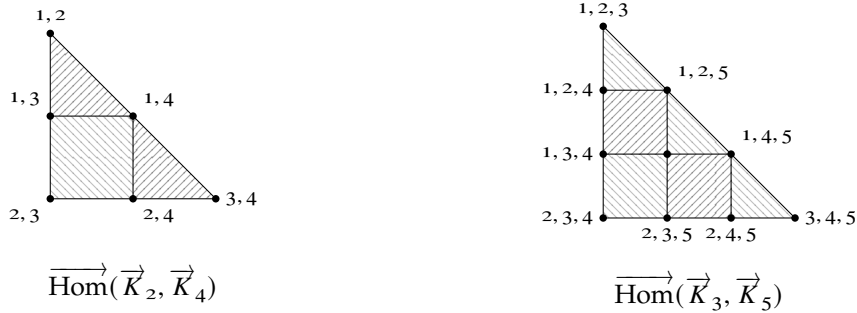
\begin{figure}[H]
\begin{subfigure}[]{0.5\textwidth}
    \centering
    \begin{tikzpicture}[scale=1.1]
   \pattern[pattern=north west lines, opacity=0.5] (0,0)--(1,0)--(1,-1)--(0,-1)--cycle;
   \pattern[pattern=north east lines, opacity=0.6] (0,0)--(1,0)--(0,1)--cycle;
   \pattern[pattern=north east lines, opacity=0.6] (2,-1)--(1,0)--(1,-1)--cycle;
   
   \path[draw] (0,0)node[dot] {} node[anchor=south east] {\tiny{$1,3$}} -- (1,0)node[dot] {} node[anchor=south west] {\tiny{$1,4$}} -- (1,-1) node[dot] {} node[anchor=north west] {\tiny{$2,4$}} -- (0,-1) node[dot] {} node[anchor=north east] {\tiny{$2,3$}}--cycle;
   \draw[] (0,0)--(0,1)node[dot] {} node[anchor=south] {\tiny{$1,2$}} --(1,0)-- (2,-1)node[dot] {} node[anchor=west] {\tiny{$3,4$}} --(1,-1);
   \end{tikzpicture}\vspace{0.25cm}
    \caption*{$\Hom(\overrightarrow{K}_{2},\overrightarrow{K}_{4})$}\label{fig:homk2k4}
\end{subfigure}
\begin{subfigure}[]{0.4\textwidth}
    \centering
    \begin{tikzpicture}[scale=0.85]
   \pattern[pattern=north west lines,opacity=0.5] (0,0) --(0,1) --(1,1)--(1,0) --cycle;
   
   \pattern[pattern=north east lines,opacity=0.6] (0,1)-- (1,1)--(1,2) --(0,2) --cycle;
   \pattern[pattern=north east lines,opacity=0.6] (1,1) --(1,0)-- (2,0) --(2,1) -- cycle;
   
   
   \pattern[pattern=north west lines,opacity=0.5] (1,2)--(1,1)--(2,1)-- cycle;
   
   \pattern[pattern=north west lines,opacity=0.5] (0,3)--(0,2)--(1,2)-- cycle;
   
   \pattern[pattern=north west lines,opacity=0.5] (2,0)--(3,0)--(2,1)-- cycle;
   
   \draw[] (0,0) node[dot] {} node[anchor=east] {\tiny{$2,3,4$}}--(1,0) node[dot] {} node[anchor=north]  {\tiny{$2,3,5$}} --(2,0)node[dot] {} node[anchor=north] {\tiny{$2,4,5$}}-- (3,0)node[dot] {} node[anchor=west] {\tiny{$3,4,5$}} --(2,1) node[dot] {} node[anchor=south west] {\tiny{$1,4,5$}}--(1,2)node[dot] {} node[anchor=south west] {\tiny{$1,2,5$}}--(0,3)node[dot] {} node[anchor=south] {\tiny{$1,2,3$}}--(0,2)--(0,1)--cycle;
   
   \draw (0,2)node[dot] {} node[anchor=east] {\tiny{$1,2,4$}} --(1,2)--(1,1)--(1,0);
   \draw (0,1)node[dot] {} node[anchor=east] {\tiny{$1,3,4$}}--(1,1) node[dot] {}--(2,1)--(2,0);
   \end{tikzpicture}\vspace{0.1cm}
    \caption*{$\Hom(\overrightarrow{K}_{3},\overrightarrow{K}_{5})$}\label{fig:homk3k5}
\end{subfigure}
   \caption{Examples of $\Hom$ complexes between acyclic tournaments.}
   \label{fig:example1}
\end{figure}

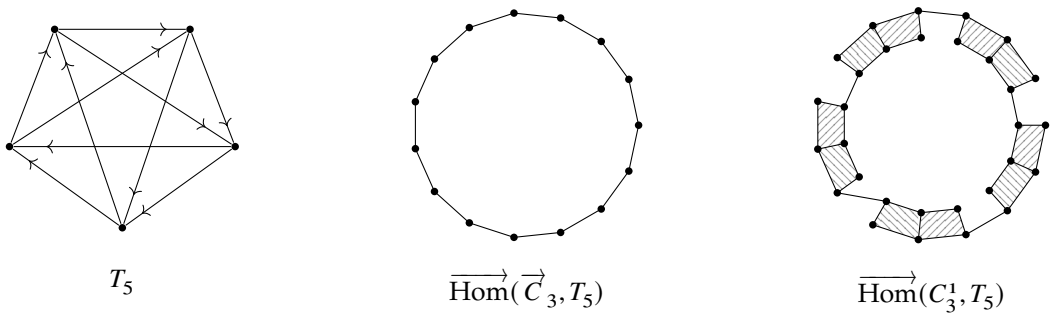
\begin{figure}[H]
\begin{subfigure}[]{0.32\textwidth}
    \centering
    \begin{tikzpicture}[scale=1.2]
\node[dot] (a) at (0,0) {};
\node[dot] (b) at (1.25,0.9) {};
\node[dot] (c) at (0.75,2.2) {};
\node[dot] (d) at (-0.75,2.2) {};
\node[dot] (e) at (-1.25,0.9) {};
 \draw[->-=.85] (b)--(a);
 \draw[->-=.85] (c)--(a);
 \draw[->-=.85] (a)--(d);
 \draw[->-=.85] (a)--(e);
 \draw[->-=.85] (c)--(b);
 \draw[->-=.85] (d)--(b);
 \draw[->-=.85] (b)--(e);
 \draw[->-=.85] (d)--(c);
 \draw[->-=.85] (e)--(c);
 \draw[->-=.85] (e)--(d);
\end{tikzpicture}\vspace{0.2cm}
    \caption*{$T_5$}\label{fig:T5}
\end{subfigure}
\begin{subfigure}[]{0.32\textwidth}
    \centering
    \begin{tikzpicture}[scale=1]
   \newdimen\R
    \R=1.5cm
   \draw (0:\R) node[dot] {}\foreach \x in {24,48,...,359} {
            -- (\x:\R) node[dot] {}
        } -- cycle (90:\R);
   
   \end{tikzpicture}\vspace{0.2cm}
    \caption*{$\Hom(\overrightarrow{C}_{3},T_5)$}\label{fig:c3t5}
\end{subfigure}
\begin{subfigure}[]{0.32\textwidth}
    \centering
    \begin{tikzpicture}[scale=0.9]
   \newdimen\R
    \R=1.3cm
    \newdimen\S
    \S=1.7cm
   \draw (0:\R) node[dot] {} -- (24:\R) node[dot] {} -- (48:\R) node[dot] {} -- (72:\R) node[dot] {} (96:\R) node[dot] {} -- (120:\R) node[dot] {} -- (144:\R) node[dot] {} -- (168:\R) node[dot] {} -- (192:\R) node[dot] {} -- (216:\R) node[dot] {} (240:\R) node[dot] {} -- (264:\R) node[dot] {} -- (288:\R) node[dot] {} (312:\R) node[dot] {} -- (336:\R) node[dot] {} -- (359:\R);
   
   \draw (240:\S) node[dot] {} -- (264:\S) node[dot] {} -- (288:\S) node[dot] {} -- (312:\S) node[dot] {} -- (336:\S) node[dot] {} -- (0:\S) node[dot] {} (24:\S) node[dot] {} -- (48:\S) node[dot] {} -- (72:\S) node[dot] {} -- (96:\S) node[dot] {} -- (120:\S) node[dot] {} -- (144:\S) node[dot] {} (168:\S) node[dot] {} -- (192:\S) node[dot] {} -- (216:\S) node[dot] {};
   
   \draw \foreach \x in {0,24,48,...,336} {
            (\x:\R)-- (\x:\S) } (216:\S) -- (240:\R);
   
   \pattern[pattern=north west lines,opacity=0.6] (24:\R)--(24:\S)--(48:\S)--(48:\R)--cycle;
   \pattern[pattern=north east lines,opacity=0.6] (48:\R)--(48:\S)--(72:\S)--(72:\R)--cycle;
   \pattern[pattern=north east lines,opacity=0.6] (96:\R)--(96:\S)--(120:\S)--(120:\R)--cycle;
   \pattern[pattern=north west lines,opacity=0.6] (120:\R)--(120:\S)--(144:\S)--(144:\R)--cycle;
   \pattern[pattern=north east lines,opacity=0.6] (168:\R)--(168:\S)--(192:\S)--(192:\R)--cycle;
   \pattern[pattern=north west lines,opacity=0.6] (192:\R)--(192:\S)--(216:\S)--(216:\R)--cycle;
   \pattern[pattern=north west lines,opacity=0.6] (240:\R)--(240:\S)--(264:\S)--(262:\R)--cycle;
   \pattern[pattern=north east lines,opacity=0.6] (264:\R)--(264:\S)--(288:\S)--(288:\R)--cycle;
   \pattern[pattern=north west lines,opacity=0.6] (312:\R)--(312:\S)--(336:\S)--(336:\R)--cycle;
   \pattern[pattern=north east lines,opacity=0.6] (336:\R)--(336:\S)--(0:\S)--(0:\R)--cycle;
   
   \end{tikzpicture}\vspace{0.2cm}
    \caption*{$\Hom(C_3^1,T_5)$}\label{fig:c4t5}
\end{subfigure}
\caption{The tournament $T_5$ and two of its $\Hom$ complexes.}
   \label{fig:example2}
\end{figure}    

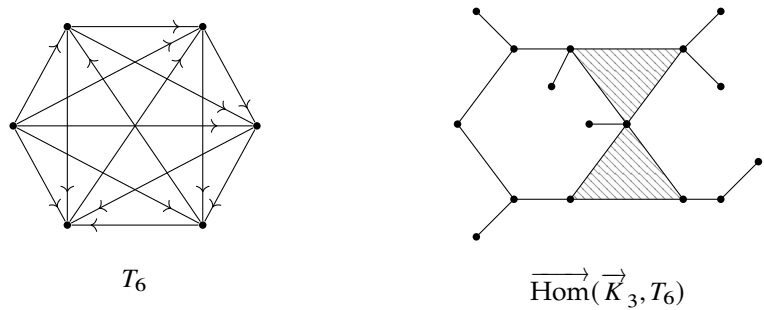
\begin{figure}[H]
\begin{subfigure}[]{0.4\textwidth}
    \centering
    \begin{tikzpicture}[scale=1.2]
\node[dot] (a) at (-0.75,0) {};
\node[dot] (b) at (0.75,0) {};
\node[dot] (c) at (1.35,1.1) {};
\node[dot] (d) at (0.75,2.2) {};
\node[dot] (e) at (-0.75,2.2) {};
\node[dot] (f) at (-1.35,1.1) {};
 \draw[->-=.85] (a)--(d);
 \draw[->-=.85] (b)--(a);
 \draw[->-=.85] (c)--(a);
 \draw[->-=.85] (e)--(a);
 \draw[->-=.85] (f)--(a);
 \draw[->-=.85] (b)--(e);
 \draw[->-=.85] (c)--(b);
 \draw[->-=.85] (d)--(b);
 \draw[->-=.85] (f)--(b);
 \draw[->-=.85] (d)--(c);
 \draw[->-=.85] (e)--(c);
 \draw[->-=.85] (f)--(c);
 \draw[->-=.85] (e)--(d);
 \draw[->-=.85] (f)--(d);
 \draw[->-=.85] (f)--(e);
\end{tikzpicture}\vspace{0.2cm}
    \caption*{$T_6$}\label{fig:T6}
\end{subfigure}
\begin{subfigure}[]{0.35\textwidth}
    \centering
    \begin{tikzpicture}[scale=1]
   \pattern[pattern=north west lines,opacity=0.6] (0,0) -- (1.5,0)  -- (0.75,1) --cycle;
   \pattern[pattern=north west lines,opacity=0.6] (0,2) -- (1.5,2)  -- (0.75,1) --cycle;
   
   \draw (0,0) node[dot] {}  -- (1.5,0) node[dot] {}  -- (0.75,1) node[dot] {}  --cycle;
   \draw (0,2) node[dot]{}-- (1.5,2) node[dot]{}  -- (0.75,1)node[dot]{} --cycle;
   \draw (2,2.5) node[dot]{}-- (1.5,2) -- (2,1.5) node[dot]{};
   \draw (0.25,1) node[dot]{}-- (0.75,1)  (-0.25,1.5) node[dot]{} -- (0,2)-- (-0.75,2) node[dot]{} -- (-1.5, 1) node[dot]{} -- (-0.75,0) node[dot]{}-- (0,0) (1.5,0)-- (2,0) node[dot]{}-- (2.5,0.5) node[dot]{};
   \draw (-0.75,2)-- (-1.25,2.5) node[dot]{} (-0.75,0)-- (-1.25,-0.5) node[dot]{};
   \end{tikzpicture}\vspace{0.2cm}
    \caption*{$\Hom(\overrightarrow{K}_{3},T_6)$}\label{fig:k3t6}
    \end{subfigure}
    \caption{The tournament $T_6$ and its complex of morphisms from the acyclic $3$-tournament.}\label{fig:example3}
    \end{figure}
    
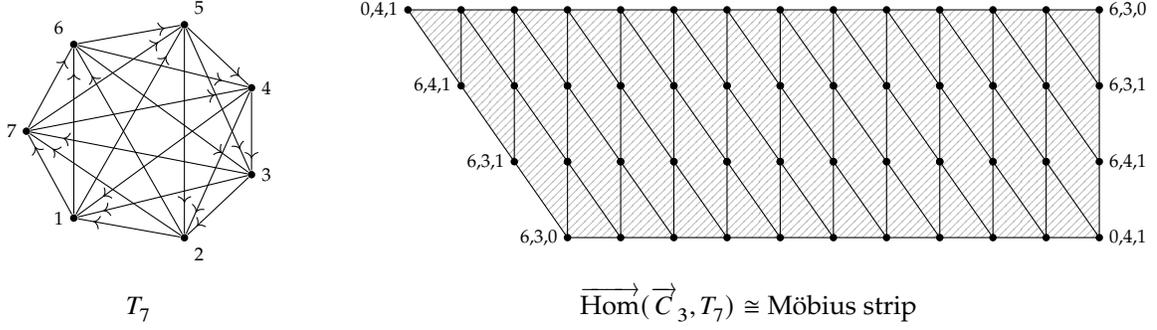
\begin{figure}[H]
\begin{subfigure}[]{0.33\textwidth}
    \centering
    \begin{tikzpicture}[scale=1.05]
\node[dot] (a) at (-0.75,0) {};
\node[dot] (b) at (0.65,-0.25) {};
\node[dot] (c) at (1.5,0.55) {};
\node[dot] (d) at (1.5,1.65) {};
\node[dot] (e) at (0.65,2.45) {};
\node[dot] (f) at (-0.75,2.2) {};
\node[dot] (g) at (-1.35,1.1) {};
 \draw[->-=.85] (a) node[anchor=east] {\tiny{1}} --(e);
 \draw[->-=.85] (a)--(f);
 \draw[->-=.85] (a)--(g);
 \draw[->-=.85] (b) node[anchor= north west] {\tiny{2}}--(a); 
 \draw[->-=.85] (c) node[anchor=west] {\tiny{3}} --(a);
 \draw[->-=.85] (d) node[anchor=west] {\tiny{4}}--(a);
 \draw[->-=.85] (c)--(b);
 \draw[->-=.85] (d)--(b);
 \draw[->-=.85] (e) node[anchor=south west] {\tiny{5}} --(b);
 \draw[->-=.85] (b)--(f);
 \draw[->-=.85] (b)--(g);
 \draw[->-=.85] (d)--(c);
 \draw[->-=.85] (e)--(c);
 \draw[->-=.85] (f) node[anchor=south east] {\tiny{6}} --(c);
 \draw[->-=.85] (c)--(g) node[anchor=east] {\tiny{7}} ;
 \draw[->-=.85] (e)--(d);
 \draw[->-=.85] (f)--(d);
 \draw[->-=.85] (g)--(d);
 \draw[->-=.85] (f)--(e);
 \draw[->-=.85] (g)--(e);
 \draw[->-=.85] (g)--(f);
\end{tikzpicture}\vspace{0.05cm}
    \caption*{$T_7$}\label{fig:t71}
\end{subfigure}
\begin{subfigure}[]{0.64\textwidth}
    \centering
    \begin{tikzpicture}[scale=1.01]
    
    \pattern[pattern=north east lines,opacity=0.5] (-0.7,2) -- (1.4,-1)  -- (8.4,-1)--(8.4,2) --cycle;
    
   \draw (-0.7,2) node[dot] {} node[anchor=east] {\tiny{1,5,2}} \foreach \x in {0,0.7,...,7.7} {
            -- (\x,2) node[dot] {}   } -- (8.4,2) node[dot] {} node[anchor=west] {\tiny{7,4,1}};
   \draw (1.4,-1) node[dot] {} node[anchor=east] {\tiny{7,4,1}} \foreach \x in {2.1,2.8,...,7.7} {
            -- (\x,-1) node[dot] {}   }-- (8.4,-1) node[dot] {} node[anchor=west] {\tiny{1,5,2}};
            
    \draw \foreach \x in {1.4,2.1,2.8,...,8.4} {  (\x,-1)  -- (\x,0) node[dot] {} -- (\x,1) node[dot] {}--(\x,2)
       (\x,-1)--(\x-0.7,0) node[dot] {} --(\x-1.4,1) node[dot] {} --(\x-2.1,2) };
    \draw (0,1) node[anchor=east] {\tiny{7,5,2}} --(0,2) (0.7,0) node[anchor=east] {\tiny{7,4,2}}--(0.7,1)--(0.7,2) (8.4,0) node[anchor=west] {\tiny{7,5,2}} --(7.7,1)--(7.0,2) (8.4,1) node[anchor=west] {\tiny{7,4,2}}--(7.7,2);
   \end{tikzpicture}\vspace{0.2cm}
    \caption*{$\Hom(\overrightarrow{C}_{3},T_7)\cong$ M{\"o}bius strip}\label{fig:c3t71}
\end{subfigure}
   \caption{The tournament $T_7$ and its complex of morphisms from the $3$-cycle $\protect\overrightarrow{C}_3$.}\label{fig:example4}
\end{figure}

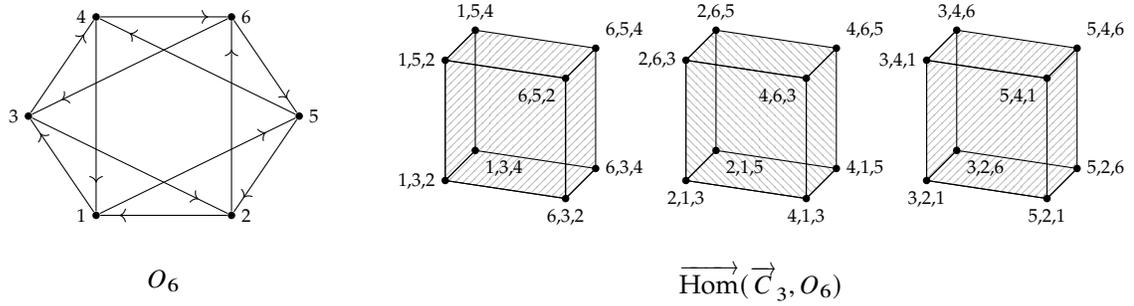
\begin{figure}[H]
\begin{subfigure}[]{0.35\textwidth}
    \centering
    \begin{tikzpicture}[scale=1.2]
\node[dot] (a) at (-0.75,0) {};
\node[dot] (b) at (0.75,0) {};
\node[dot] (c) at (1.5,1.1) {};
\node[dot] (d) at (0.75,2.2) {};
\node[dot] (e) at (-0.75,2.2) {};
\node[dot] (f) at (-1.5,1.1) {};
 \draw[->-=.85] (b)node[anchor=west] {\tiny{2}} --(a) node[anchor=east] {\tiny{1}};
 \draw[->-=.85] (e) node[anchor=east] {\tiny{4}} --(a);
 \draw[->-=.85] (a) -- (f) node[anchor=east] {\tiny{3}} ;
 \draw[->-=.85] (a)--(c) node[anchor=west] {\tiny{5}};
 \draw[->-=.85] (c)--(b);
 \draw[->-=.85] (b)--(d) node[anchor=west] {\tiny{6}};
 \draw[->-=.85] (f)--(b);
 \draw[->-=.85] (d)--(c);
 \draw[->-=.85] (c)--(e);
 \draw[->-=.85] (e)--(d);
 \draw[->-=.85] (d)--(f);
 \draw[->-=.85] (f)--(e);
\end{tikzpicture}\vspace{0.2cm}
    \caption*{$O_6$}\label{fig:o6}
\end{subfigure}
\begin{subfigure}[]{0.6\textwidth}
    \centering
    \begin{tikzpicture}[scale=0.8]
    \pattern[pattern=north east lines,opacity=0.5] (0,0) -- (2,-0.3) -- (2,1.7) -- (0,2) -- cycle;
    \pattern[pattern=north east lines,opacity=0.5] (0,2) -- (2,1.7) -- (2.5,2.2) -- (0.5,2.5) -- cycle;
    \pattern[pattern=north east lines,opacity=0.5] (2,-0.3) -- (2,1.7) -- (2.5,2.2) -- (2.5,0.2) -- cycle;
    
    \draw (0,0) node[dot]{} node[anchor=east] {\tiny{1,3,2}} -- (0,2) node[dot]{}node[anchor=east] {\tiny{1,5,2}}  -- (0.5,2.5)node[dot]{} node[anchor=south] {\tiny{1,5,4}}  -- (0.5,0.5)node[dot]{} node[anchor=north west] {\tiny{1,3,4}}  -- cycle;
    \draw (0,0) -- (2,-0.3) node[dot]{} node[anchor=north] {\tiny{6,3,2}} -- (2,1.7)node[dot]{} node[anchor=north east] {\tiny{6,5,2}} -- (0,2) -- cycle;
    \draw (0,0) -- (2,-0.3) -- (2.5,0.2) node[dot]{} node[anchor=west] {\tiny{6,3,4}} -- (0.5,0.5) -- cycle;[fill=green,fill opacity=0.5] (0.5,0.5) -- (0.5,2.5) -- (2.5,2.2)  -- (2.5,0.2) -- cycle;
    \draw (2,-0.3) -- (2,1.7) -- (2.5,2.2) -- (2.5,0.2) -- cycle;
    \draw (0,2) -- (2,1.7) -- (2.5,2.2)node[dot]{} node[anchor= south west] {\tiny{6,5,4}} -- (0.5,2.5) -- cycle;

    \pattern[pattern=north west lines,opacity=0.5] (4,0) -- (6,-0.3) -- (6,1.7) -- (4,2) -- cycle;
    \pattern[pattern=north west lines,opacity=0.5] (4,2) -- (6,1.7) -- (6.5,2.2) -- (4.5,2.5) -- cycle;
    \pattern[pattern=north west lines,opacity=0.5] (6,-0.3) -- (6,1.7) -- (6.5,2.2) -- (6.5,0.2) -- cycle;
    
    \draw (4,0) node[dot]{} node[anchor=north] {\tiny{2,1,3}} -- (4,2) node[dot]{}node[anchor=east] {\tiny{2,6,3}}  -- (4.5,2.5)node[dot]{} node[anchor=south] {\tiny{2,6,5}}  -- (4.5,0.5)node[dot]{} node[anchor=north west] {\tiny{2,1,5}}  -- cycle;
    \draw (4,0) -- (6,-0.3) node[dot]{} node[anchor=north] {\tiny{4,1,3}} -- (6,1.7)node[dot]{} node[anchor=north east] {\tiny{4,6,3}} -- (4,2) -- cycle;
    \draw (4,0) -- (6,-0.3) -- (6.5,0.2) node[dot]{} node[anchor=west] {\tiny{4,1,5}} -- (4.5,0.5) -- cycle;[fill=green,fill opacity=0.5] (4.5,0.5) -- (4.5,2.5) -- (6.5,2.2)  -- (6.5,0.2) -- cycle;
    \draw (6,-0.3) -- (6,1.7) -- (6.5,2.2) -- (6.5,0.2) -- cycle;
    \draw (4,2) -- (6,1.7) -- (6.5,2.2)node[dot]{} node[anchor= south west] {\tiny{4,6,5}} -- (4.5,2.5) -- cycle;

    \pattern[pattern=north east lines,opacity=0.5] (8,0) -- (10,-0.3) -- (10,1.7) -- (8,2) -- cycle;
    \pattern[pattern=north east lines,opacity=0.5] (8,2) -- (10,1.7) -- (10.5,2.2) -- (8.5,2.5) -- cycle;
    \pattern[pattern=north east lines,opacity=0.5] (10,-0.3) -- (10,1.7) -- (10.5,2.2) -- (10.5,0.2) -- cycle;
    
    \draw (8,0) node[dot]{} node[anchor=north] {\tiny{3,2,1}} -- (8,2) node[dot]{}node[anchor=east] {\tiny{3,4,1}}  -- (8.5,2.5)node[dot]{} node[anchor=south] {\tiny{3,4,6}}  -- (8.5,0.5)node[dot]{} node[anchor=north west] {\tiny{3,2,6}}  -- cycle;
    \draw (8,0) -- (10,-0.3) node[dot]{} node[anchor=north] {\tiny{5,2,1}} -- (10,1.7)node[dot]{} node[anchor=north east] {\tiny{5,4,1}} -- (8,2) -- cycle;
    \draw (8,0) -- (10,-0.3) -- (10.5,0.2) node[dot]{} node[anchor=west] {\tiny{5,2,6}} -- (8.5,0.5) -- cycle;[fill=green,fill opacity=0.5] (8.5,0.5) -- (8.5,2.5) -- (10.5,2.2)  -- (10.5,0.2) -- cycle;
    \draw (10,-0.3) -- (10,1.7) -- (10.5,2.2) -- (10.5,0.2) -- cycle;
    \draw (8,2) -- (10,1.7) -- (10.5,2.2)node[dot]{} node[anchor= south west] {\tiny{5,4,6}} -- (8.5,2.5) -- cycle;
   \end{tikzpicture}\vspace{0.2cm}
    \caption*{$\Hom(\overrightarrow{C}_{3},O_6)$}\label{fig:k3o6}
    \end{subfigure}
    \caption{The octahedral graph \texorpdfstring{$O_6$}{oct} and its complex of morphisms from the $3$-cycle $\protect\overrightarrow{C}_3$.}
    \label{fig:example5}
    \end{figure}

\section{Structural properties of \texorpdfstring{$\Hom$}{dh} complexes} \label{sec:structural}

In this section we collect some structural results regarding directed homomorphism complexes, many of which extend analogous properties for Hom complexes of undirected graphs.  We first point out some immediate properties of the $\Hom$ complexes, first established in the context of undirected graphs by Babson and Kozlov in \cite{BabKozCom}.

\begin{itemize}
    \item For any two digraphs $G$ and $H$, we have that $\Hom(G,H)$ is a polyhedral complex (and hence a regular CW complex);
    \smallskip
    
    \item The cells of $\Hom(G,H)$ are products of simplices;
        \smallskip
    \item
    The map $\Hom(G,-)$ is a covariant, while $\Hom(-,H)$ is a contravariant, functor from the category of directed graphs to the category of topological spaces.  If $f \in \Hom_0(A,B)$ is a graph homomorphism, we let 
    \[f_*: \Hom(G,A) \rightarrow \Hom(G,B), \hspace{.2 in} f^*:\Hom(B,H) \rightarrow \Hom(A,H)\]
    
    \noindent
    denote the induced cellular (poset) maps;
        \smallskip
    \item
In particular, if a group $\Gamma$ acts on a digraph $G$ then for any digraph $H$, $\Gamma$ acts on the complex $\Hom(G,H)$;

        \smallskip
    \item
    The map induced by composition
    \[\Hom(G,H) \times \Hom(H,K) \rightarrow \Hom(G,K)\]
    is a topological (poset) map.
    \end{itemize}

Recall that for digraphs $G$ and $H$ the \defi{coproduct} $G \sqcup H$ is the digraph whose vertex set and edge set are given by disjoint unions.  The homomorphism complex interacts with this operation in a predictable way.

\begin{prop}\label{prop:coprod}
For any digraphs $A$, $B$, $C$ we have an isomorphism of posets
\[ \Hom(A \sqcup B, C) \cong \Hom(A, C) \times \Hom(B,C).\]
If $A$ is connected and has at least one edge, we have an isomorphism of posets
\[ \Hom(A , B \sqcup C) \cong \Hom(A, B) \sqcup \Hom(A,C).\]
\end{prop}
\begin{proof}
The first formula follows from the fact that any multihomomorphism $\alpha: A \sqcup B \rightarrow C$ is determined by its restriction to $A$ and $B$.  For the second formula, suppose that $\beta: V(A) \rightarrow 2^{V(B) \sqcup V(C)} \backslash \{\emptyset\}$ and $x \sim y$ in $A$.  Then if $\beta(x) \cap V(B) \neq \emptyset$ then we must have $\beta(y) \subset V(B)$. Since $A$ is connected, it follows that $\cup_{x \in V(A)} \beta(x) \subset V(B)$.
\end{proof}

As is the case with Hom complexes of undirected graphs the $\Hom$ construction also interacts well with products in the second coordinate.

\begin{prop}\label{prop:product}
Suppose $T$, $G$, and $H$ are digraphs.  Then we have a natural strong homotopy equivalence of posets
\[ \Hom(T, G \times H) \rightarrow \Hom(T, G) \times \Hom(T,H).\]
\end{prop}

\begin{proof}
The proof is more or less identical to the proof of Proposition 3.8 from \cite{DocHom}, we sketch the argument here. Let $P = \Hom(T, G \times H)$ and $Q = \Hom(T,G) \times \Hom(T,H)$ denote the respective posets. 
First define a poset map $i:Q \rightarrow P$ by $i(\alpha, \beta)(v) = \alpha(v) \times \beta(v)$, for any $(\alpha, \beta) \in Q$ and any vertex $v \in V(T)$.  It is clear that $i$ is injective.

Next define a poset map $c:P \rightarrow Q$ as follows. If $\gamma \in P$ and $v \in V(T)$, we have minimal vertex subsets $A_v \subseteq V(G)$, $B_v \subseteq V(H)$ such that $\gamma(v) \subseteq \{(a, b): a \in A_v, b \in B_v\}$.  Define $c(\gamma) = \alpha \times \beta$, where $\alpha(v) = A_v$ and $\beta(v) = B_v$ are these subsets.

One can check that these are well-defined poset maps and also that $c \circ i = \id_Q$ and $i \circ c \geq \id_P$, from which the result follows.
\end{proof}

We will see that $\Hom$ complexes also behave well with respect to the internal hom adjunction described in Section \ref{sec:definitions}, as well as a notion of `folding' for digraphs.  Again this mimics the situation for undirected graphs and, as pointed out by Matsushita in \cite{MatMor}, applies in more general contexts as well.  We will import some results from \cite{MatMor} that can applied to digraphs.  For this we first review some relevant concepts.

For a fixed positive integer $r$, an \defi{$r$-set} is a pair $X = (V(X), R(X))$ consisting of a set $V(X)$ and a subset $R(X)$ of $V(X)^r$, the $r$-times direct product of $V(X)$.  Note that a directed graph is simply an $r$-set with $r=2$, whereas an $r$-uniform hypergraph is an $r$-set whose relation set $R(X)$ is closed under the $S_r$-action on $V(X)^r$.

In \cite{MatMor} Matsushita develops a theory of homomorphism complexes of $r$-sets that specializes to the $\Hom$ complexes of digraphs discussed here.  He establishes several properties of such complexes, in many cases extending the proofs given for homomoporhism complexes of undirected graphs given by the first author in \cite{DocHom}.  In particular, we have the following.

\begin{prop}\label{prop:adjoint} \cite[Lemma 4.4]{MatMor}
Suppose $A$, $B$, $C$ are digraphs.  Then we have a natural strong homotopy equivalence of posets
\[ \Hom(A \times B, C) \rightarrow \Hom(A, C^B).\]
\end{prop}

Let ${\bf 1}$ denote the graph with a single vertex $\{v\}$ and a single edge (loop) $(v,v)$. Note that for any digraph $G$ we have $G = G \times {\bf 1} \cong G \cong {\bf 1} \times G$, so that from Proposition \ref{prop:adjoint} we get a strong homotopy equivalence
\[ \Hom(G,H) \cong \Hom({\bf 1} \times G, H) \simeq \Hom({\bf 1}, H^G).\]

\begin{rem}\label{rem:clique}
The complex $\Hom({\bf 1}, H^G)$ can alternatively be described as follows.  Form an undirected graph ${\mathcal G}(H^G)$ on the looped vertices of $H^G$ with adjacency $f \sim g$ if both directed edges $(f,g)$ and $(g,f)$ exist in the digraph $H^G$.  Then $\Hom({\bf 1}, H^G)$ coincides with $X({\mathcal G}(H^G))$, the clique complex of the graph ${\mathcal G}(H^G)$.  Hence we have $\Hom(G,H) \simeq X({\mathcal G}(H^G))$.
\end{rem}

\begin{rem}
One can use Proposition \ref{prop:adjoint} and Remark \ref{rem:clique} to prove that the functor $\Hom(T,-)$ preserves finite limits, in the sense that if $D$ is any finite diagram of digraphs with limit $\lim D$ we have a homotopy equivalence
\begin{displaymath}
\big|\Hom\big(T,\lim (D)\big)\big| \simeq \big|\lim \big(\Hom(T,
D)\big)\big|.
\end{displaymath}

This follows from the fact that taking a clique complex of a graph preserves limits, and that one can view ${\mathcal G}(-)$ as a functor to the category of reflexive undirected graphs that is right adjoint to inclusion.  Details can be found in Proposition 3.7 of \cite{DocHom}.  
\end{rem}

In the case of undirected graphs, the notion of a \emph{folding} proved useful in analyzing the homotopy type of $\oldHom$ complexes, for instance in showing that $\oldHom(T,G) \simeq \oldHom(K_2, G)$ for any connected tree $T$ with at least one edge \cite{BabKozCom}.  In \cite{KozSim} Kozlov provided a simple proof that foldings preserve homotopy type in both entries of the $\oldHom$ complex of undirected graphs. As was established by Matsushita a similar result holds for $r$-sets and hence can be applied to $\Hom$ complexes of digraphs.  Let us first set some notation.

\begin{defn}\label{def:folding}
Suppose $G$ is a digraph  with vertices $v$ and $w$ such that $\inN_G(v) \subset \inN_G(w)$ and $\outN_G(v) \subset \outN_G(w)$.  Then the vertex $v$ is said to be \defi{dismantlable}, and the mapping $v \mapsto w$ defines a graph homomorphism $G \rightarrow G \backslash \{v\}$ called a \defi{directed folding}.  A digraph $G$ is \defi{dismantlable} if there exists a sequence of directed foldings that results in the single looped vertex ${\bf 1}$.
\end{defn}

Results from \cite{MatMor} can then be applied to obtain the following.

\begin{prop} \label{prop:folding} \cite[Theorem 5.6]{MatMor}
Let $G$ be a digraph and suppose $f: G \rightarrow G \backslash \{v\}$ is a directed folding. Then for any digraph $H$ we have natural strong homotopy equivalences of posets
\[f_*: \Hom(H, G) \rightarrow \Hom(H,G\backslash \{v\}),\]
\[f^*:\Hom(G \backslash \{v\}, H) \rightarrow \Hom(G, H).\]
\end{prop}

We note that Matsushita's proof of this result for $r$-sets relies on an auxiliary result regarding \emph{strong homotopy} of $r$-sets.  We further discuss this notion in Section \ref{sec:homotopy}.

\begin{example}
As an example of the folding result we refer to Figure \ref{fig:example2}.  Here we see that the vertex $4$ of the graph $G = C^1_3$ satisfies $\outN_G(4) \subset \outN_G(3)$ and $\inN_G(4) = \emptyset \subset \inN_G(3)$. Hence the folding $4 \mapsto 3$ induces a homotopy equivalence of the relevant $\Hom$ complexes.
\end{example}

\begin{rem}
In the context of digraphs, foldings exhibit some perhaps unexpected behavior.  For undirected graphs, a path of arbitrary length (in fact any finite tree) can be folded down to a single edge, so that all paths are in some sense equivalent as far as the Hom complexes are concerned.  In the directed case this is no longer the case, as a directed path $1 \rightarrow 2 \rightarrow \cdots \rightarrow n$ has no dismantlable vertex.  On the other hand we note that a path $1 \rightarrow 2 \leftarrow 3 \rightarrow 4 \leftarrow \cdots n$ consisting of alternating sinks and sources can be folded down to a single directed edge.  We discuss applications of this in Section \ref{sec:homotopy}.
\end{rem}

\section{Neighborhood complexes and other constructions}\label{sec:neighbor}

In his original study of topological bounds on chromatic number, Lov\'asz \cite{Lov} introduced the \defi{neighborhood complex} ${\mathcal N}(G)$ of an undirected graph $G$, by definition the simplicial complex on the vertices of $G$ where a face is given by all subsets that have a common neighbor. It was later shown by Babson and Kozlov \cite{BabKozCom} that the neighborhood complex could be recovered as a homomorphism complex, in the sense that there exists a homotopy equivalence ${\mathcal N}(G) \simeq \oldHom(K_2,G)$.  We will see that a similar result holds in the directed setting.  We start with a definition.

\begin{defn}\label{def:neighbor}
Suppose $G$ is a digraph.  The \defi{out-neighborhood complex} $\N(G)$ is the simplicial complex on vertex set $\{v \in V(G): \mathrm{indeg}(v)>0\}$, with facets given by the out neighborhoods $\outN_G(v)$ for all $v \in V(G)$.
\end{defn}

\begin{figure}[H]
\begin{subfigure}[]{0.3\textwidth}
    \centering
    \begin{tikzpicture}
\node (a) at (0,0) {1};
\node (b) at (0,-2) {2};
\node (c) at (2,0) {3};
\node (d) at (2,-2) {4};
\node (e) at (-1,-1.25) {5};
\draw[->-=.7] (a)--(b);
\draw[->-=.73] (c)--(a);
\draw[->-=.7] (c)--(b);
\draw[->-=.7] (d)--(a);
\draw[->-=.73] (c)--(d);
\draw[->-=.7] (e)--(a);
\draw[->-=.7] (e)--(c);
\draw[->-=.7] (e)--(d);
\draw[->-=.7] (b)--(d);
\end{tikzpicture}
\caption{$G$}
\end{subfigure}
\begin{subfigure}[]{0.3\textwidth}
    \centering
     \begin{tikzpicture}[scale=1.3]
    
   \pattern[pattern=north east lines,opacity=0.5] (0.0,  0.0)--(-0.9,-0.5)--(1.5, -0.5)--cycle;
   \pattern[pattern=north west lines,opacity=0.5] (-0.5, -1.5)--(-0.9,-0.5)--(1.5, -0.5)--cycle;
  \path[draw] (0.0,  0.0)--(-0.9,-0.5);
  \path[draw] (-0.9,-0.5)--(1.5, -0.5);
  \path[draw] (0.0,  0.0)--(1.5, -0.5);
  \path[draw] (-0.5, -1.5)--(-0.9,-0.5);
  \path[draw] (-0.5, -1.5)--(1.5, -0.5);
  
  \node[dot] (a) at ( 0.0,  0.0) {};
  \node[dot] (b) at (-0.5, -1.5) {};
  \node[dot] (e) at ( 1.5, -0.5) {};
  \node[dot] (f) at (-0.9,-0.5) {};
   \node[inner sep=0pt] () at ( -0.1,  0.2) {{$3$}};
   \node[inner sep=0pt] () at ( -0.7,  -1.5) {{$2$}};
   \node[inner sep=0pt] () at ( 1.5,  -0.25) {{$4$}};
   \node[inner sep=0pt] () at ( -1.1,  -0.5) {{$1$}};
   \end{tikzpicture}
    \caption{$\N(G)$}
\end{subfigure}
\begin{subfigure}[]{0.3\textwidth}
    \centering
    \begin{tikzpicture}[scale=1.3]
    
   \pattern[pattern=north east lines,opacity=0.5] (0.0,  0.0)--(-0.5, -1.5)--(1.0, -1.5)--cycle;
   \pattern[pattern=north west lines,opacity=0.5] (0.0,  0.0)--(1.5, -0.5)--(1.0, -1.5)--cycle;
   \path[draw] ( 0.0,  0.0)--(-0.5, -1.5);
   \path[draw] (-0.5, -1.5)--(1.0, -1.5);
   \path[draw] ( 0.0,  0.0)--(1.0, -1.5);
   \path[draw] ( 0.0,  0.0)--(1.5, -0.5);
   \path[draw] (1.0, -1.5)--(1.5, -0.5);
   \path[draw] ( 0.0,  0.0)--(-0.9,-0.5);
   
   \node[dot] (a) at ( 0.0,  0.0) {};
   \node[dot] (b) at (-0.5, -1.5) {};
   \node[dot] (c) at ( 1.0, -1.5) {};
   \node[dot] (e) at ( 1.5, -0.5) {};
   \node[dot] (f) at (-0.9,-0.5) {};
   \node[inner sep=0pt] () at ( -0.1,  0.2) {{$3$}};
   \node[inner sep=0pt] () at ( -0.7,  -1.5) {{$2$}};
   \node[inner sep=0pt] () at ( 1.2,  -1.55) {{$5$}};
   \node[inner sep=0pt] () at ( 1.5,  -0.25) {{$4$}};
   \node[inner sep=0pt] () at ( -1.1,  -0.5) {{$1$}};
   \end{tikzpicture}
    \caption{$\Nin(G)$}
\end{subfigure}
   \caption{A graph $G$, along with its out- and in-neighborhood complexes.}\label{fig:exampleofnbdcomplex}
   \label{fig:outcomplex}
\end{figure}
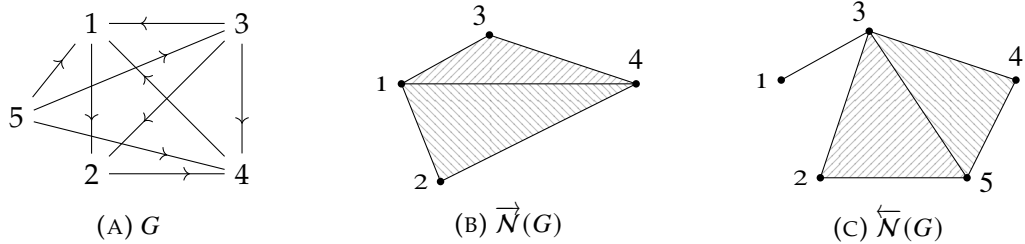

Note that vertices $v \in V(G)$ with $\indeg(v) = 0$ are not vertices of the simplicial complex $\N(G)$, but are still used to define the facet $\outN_G(v)$.

\begin{rem}
The class of directed neighborhood complex includes neighborhood complexes of undirected graphs as a special case, in the following sense.   Given an undirected graph $G$, define a digraph $\overrightarrow{G}$ with vertex set $V(\overrightarrow{G})=V(G)\times \{0\} \sqcup V(G) \times \{1\}$ and with a directed edge from $(v,0)$ to $(w,1)$ in $\overrightarrow{G}$ if and only if $vw\in E(G)$. Then one can check that $\N(\overrightarrow{G})\cong \mathcal{N}(G). $
\end{rem}

We refer to Figure \ref{fig:outcomplex} for an example. Next let ${\mathcal P}(\N(G))$ denote the face poset of the out-neighborhood complex.  We then have the following observation.

\begin{thm}\label{thm:homandnbd}
For any digraph $G$, we have a strong homotopy equivalence 
\[\Hom(K_2,G) \simeq {\mathcal P}(\N(G)).\]
In particular, $\Hom(K_2,G)$ and $\N(G)$ are homotopy equivalent.
\end{thm}

\begin{proof}
Suppose $V(K_2)=\{1,2\}$ and let $12 = (1,2)$ denote the single edge of $K_2$. Define a poset map $\varphi : \Hom(K_2,G) \rightarrow {\mathcal P}(\N(G))$ by $\varphi(\eta)=\eta(2)$. Clearly, $\eta(2)\subseteq \outN_G(v)$ for each $v \in \eta(1)\neq \emptyset$, so that $\varphi$ is well-defined.

We first show that for each $B\in {\mathcal P}(\N(G))$, the set $\varphi^{-1}(B)$ has a maximal element. For this let $B\in {\mathcal P}(\N(G))$. Since $\varphi^{-1}(B)= \{(A,B): ab\in E(G),~\forall~ a\in A , b \in B\}$, we have that $(\inN_G(B),B)$ is the maximal element of $\varphi^{-1}(B)$. Here, we use the notation $\inN_G(B)=\{a \in V(G): ab \in E(G)~\forall~ b \in B\}$.

Next we prove that for each $B \in {\mathcal P}(\N(G))$ and $(C,D)\in \varphi^{-1}({\mathcal P}(\N(G))_{\geq B})$, the poset
\[\varphi^{-1}(B)\cap \Hom(K_2,G)_{\leq (C,D)}\]
\noindent
has a maximal element. Let $B \in {\mathcal P}(\N(G))$ and $(C,D)\in  \Hom(K_2,G)$ such that $\varphi((C,D))=D \supseteq B $. Clearly,  $\inN_G(B)\supseteq \inN_G(D)\supseteq C\neq \emptyset$ and $\varphi^{-1}(B)\cap \Hom(K_2,G)_{\leq (C,D)}= \{(A,B): A\subseteq C, A\neq \emptyset\}$. Therefore, $(C,B)$ is the maximal element of $\varphi^{-1}(B)\cap \Hom(K_2,G)_{\leq (C,D)}$. The result then follows from Lemma \ref{lem:posetfiber}.
\end{proof}

For any digraph $G$, we can equally well define the \defi{in-neighborhood complex} $\overleftarrow{\mathcal{N}}(G)$ on vertex set $\{v \in V(G): \mathrm{outdeg}(v)>0\}$, with facets given by the in-neighborhoods $\inN_G(v)$ for all $v \in V(G)$.  See Figure \ref{fig:exampleofnbdcomplex} for an example.  As one might expect from Theorem \ref{thm:homandnbd} this does not give anything new.

\begin{prop}\label{prop:outin}
For any digraph $G$, we have a strong homotopy equivalence
\[{\mathcal P}(\overleftarrow{\mathcal{N}}(G)) \simeq {\mathcal P}(\overrightarrow{\mathcal{N}}(G)).\]
\end{prop}

\begin{proof}
Using similar arguments as in Theorem \ref{thm:homandnbd}, it is easy to see that the map $\psi: \Hom(K_2,G) \rightarrow {\mathcal P}(\overleftarrow{\mathcal{N}}(G))$ defined as $\psi(\eta)=\eta(1)$ is a strong homotopy equivalence.  The result then follows from the composition $\varphi \circ \psi$, where $\varphi$ is the map from the proof of Theorem \ref{thm:homandnbd}. 

We also include a proof of homotopy equivalence that uses ideas that go back to Dowker \cite{Dow}. For the given digraph $G$, define a bipartite graph $B_G$ on vertex set $V_0\sqcup V_1$, where $V_0=\{(v,0)\in V(G)\times \{0\} : \mathrm{outdeg}_G(v)>0 \}$ and $V_1=\{(w,1)\in V(G)\times \{1\} : \mathrm{indeg}_G(w)>0 \}$, and $(v,0)(w,1)\in E(B_G)$ if and only if $vw\in E(G)$. Using the notation of \cite[Theorem 10.9]{Bjo95}, we have that $\Delta_0\cong \Nin(G)$ and $\Delta_1\cong \N(G)$, and from this it follows that $\Nin(G)\simeq \N(G)$.
\end{proof}

\begin{rem}
There is yet another way to prove Theorem \ref{thm:homandnbd} and Proposition \ref{prop:outin} that we wish to sketch.  For this let $G$ and $H$ be digraphs and suppose $I \subset G$ is an independent set (a collection of vertices with no edges between them).  As in Csorba \cite{CsoThesis} we can define a complex $\Hom_I(G,H)$ consisting of all multihomorphisms $\alpha:V(G) \backslash I \rightarrow 2^{V(H)} \backslash \{\emptyset\}$ that extend to $G$.  Following the proof of \cite[Theorem 2.36]{CsoThesis} one can show that $\Hom_I(G,H)$ and $\Hom(G,H)$ are homotopy equivalent.  Now suppose $G = K_2$ is a directed edge $1 \rightarrow 2$.  On the one hand if we let $I = \{1\}$ we see that $\Hom_I(K_2,H)$ is isomorphic as a simplicial complex to $\N(H)$, whereas if we take $I = \{2\}$ we get $\Hom_I(K_2,H) \cong \Nin(H)$.
\end{rem}

In the case of undirected graphs, for any tree $T$ the neighborhood complex ${\mathcal N}(T)$ is homotopy equivalent to $0$-dimensional sphere (two isolated points) since $T$ folds down to an edge. In the case of directed trees, we have something similar.

\begin{prop}
For any directed tree $T$, the complex $\N(T)$ is homotopy equivalent to a wedge of $0$-dimensional spheres.
\end{prop}

\begin{proof}
We prove this by induction on the number of vertices of $T$. If $|V(T)|=2$, then the statement is clear.  Suppose $T$ be a directed tree on more than $2$ vertices and let $v$ be a leaf vertex in $T$, {\it i.e.} $\outdeg_T(v)+\indeg_T(v)=1$.  If $\indeg_T(v)=1$ and $(u,v)\in E(T)$, then either $\outN_T(u)=\{v\}$ or $(\{v\},\outN_T(u))$ is a free pair. In the former case $\N(T)=\N(T\backslash \{v\})\cup \{v\}$, and in the latter case the complex $\N(T)$ collapses onto $\N(T\backslash \{v\})$. In either case the result follows from induction. If $\outdeg_T(v)=1$ and $(v,u)\in E(T)$, then $\N(T)=\N(T\backslash \{v\})\cup \{u\}$ and again the result follows from induction.
\end{proof}

A natural question to ask is what other kinds of topology the directed neighborhood complexes can attain. In the undirected setting Csorba has shown \cite[Theorem 7]{CsoHom} that the homotopy types of neighborhood complexes ${\mathcal N}(G)$ of simple graphs (without loops) are precisely those of finite free ${\mathbb Z}_2$-complexes. Hence this puts restrictions on the possible homotopy types of ${\mathcal N}(G)$ (for instance ${\mathcal N}(G)$ cannot have odd Euler characteristic).  On the other hand the first author has shown in \cite{DocUni} that if one allows loops on the vertices, then for any finite connected graph $T$ (with at least one edge), any homotopy type can be achieved as $\oldHom(T,G)$ for some choice of $G$.

For the case of digraphs, there is a similar universality for the directed neighborhood complexes, even when one restricts to simple directed graphs. 

    
\begin{prop}\label{prop:universal}
Suppose $X$ is any finite simplicial complex.  Then there exists a digraph $G$ (without loops) such that $\N(G) \cong X$, an isomorphism of simplicial complexes.
\end{prop}

\begin{proof}
Suppose $X$ is a finite simplicial complex on vertex set $V$. For our digraph $G$, we take its vertex set $V(G)$ to consist of the set $V$ as well as an additional vertex $v_F$ for each facet $F \in X$.  For the edges of $G$, we construct a directed edge from the vertex $v_F$ to each vertex $v \in V$ where $v \in F$. We see that $\N(G)$ is a simplicial complex on vertex set $V$, with facets given by the facets of $X$.
\end{proof}
    
What if we restrict to neighborhood complexes of tournaments (oriented complete graphs)? Is it still true that any homotopy type can be achieved as $\N(T_n)$ for some choice of tournament $T_n$?  We address this question in Section \ref{sec:tourn}.

\subsection{Mycielskian and other constructions}\label{sec:Myc}
In the context of undirected graphs, the notion of the Mycielskian of a graph plays an important role in the study of homomorphism complexes.  To recall the definitions, let $J_n$ denote the undirected graph on vertex set $\{0,1,\dots, n\}$ with edges $\{01,12, \dots (n-1)n\}$ as well as the loop $00$. For a graph $G$ and integer $n \geq 2$, the \defi{Mycielskian} graph ${\mathcal M}_n(G)$ is defined as the quotient graph
\[{\mathcal M}_n(G) = G \times J_n / \sim,\]
where $(g,n) \sim (g^\prime,n)$ for all $g, g^\prime \in G$.  The notation $\mu(G)$ is also sometimes used to indicate the construction ${\mathcal M}_2(G)$.   

The Mycielskian construction preserves the property of being triangle-free but increases the chromatic number.  In particular, the \emph{Gr\"otzsch graph} = ${\mathcal M}_2(C_5)$ is a triangle-free graph with chromatic number 4. By applying the Mycielskian construction repeatedly to a triangle-free starting graph, Mycielski showed in \cite{Myc} that there exist triangle-free graphs with arbitrarily large chromatic number.

 In \cite[Theorem 3.1]{Cso} Csorba proved that for every (undirected) graph $G$ and every $r \geq 1$, the homomorphism complex $\oldHom(K_2, {\mathcal M}_r(G))$ is ${\mathbb Z}_2$-homotopy equivalent to $S (\oldHom(K_2, G))$, the suspension of $\oldHom(K_2,G)$. We next discuss a version of the Mycielskian construction for digraphs, where now several choices can be made regarding orientations.

\begin{defn}[Directed Mycielskian]\label{def:Mycielskian}
Define the directed graphs $I^1$, $I^2$, and $I^3$ as in Figure \ref{fig1}.  For a digraph $G$, we define following.
\begin{equation}\label{eq:mycielski}
    \begin{split}
        M^1(G) & =G\times I^1 / \{(g,2)\sim (h,2) ~~ \forall g,h\in V(G)\}, \\
        M^2(G) & =G\times I^2 / \{(g,2)\sim (h,2) ~~ \forall g,h\in V(G)\}, \text{ and }\\
        M^3(G) & =G\times I^3 / \{(g,2)\sim (h,2) \text{ and } (g,0)\sim (h,0)~~ \forall g,h\in V(G)\}.
    \end{split}
\end{equation}
\end{defn}
    \begin{figure}[H]
\begin{subfigure}[]{0.3\textwidth}
    \centering
    \begin{tikzpicture}[scale=0.85]
\node (a) at (0,0) {0};
\node (b) at (1.5,0) {1};
\node (c) at (3,0) {2};
\draw[->-=.5] (a) to [out=120,in=60,looseness=10] (a);
\draw[->-=.5] (a)--(b) node[midway,above right] {};
\draw[->-=.5] (b)--(c) node[midway,above right] {};
\end{tikzpicture}
\caption{$I^1$}
\end{subfigure}
\begin{subfigure}[]{0.3\textwidth}
    \centering
    \begin{tikzpicture}[scale=0.85]
\node (a) at (0,0) {0};
\node (b) at (1.5,0) {1};
\node (c) at (3,0) {2};
\draw[->-=.5] (a) to [out=120,in=60,looseness=10] (a);
\draw[->- =.5] (b)--(a) node[midway,above right] {};
\draw[->-=.5] (c)--(b) node[midway,above right] {};
\end{tikzpicture}
\caption{$I^2$}
\end{subfigure}
\begin{subfigure}[]{0.3\textwidth}
    \centering
    \begin{tikzpicture}[scale=0.85]
\node (a) at (0,0) {0};
\node (b) at (1.5,0) {1};
\node (c) at (3,0) {2};
\draw[->-=.5] (b) to [out=120,in=60,looseness=10] (b);
\draw[->- =.5] (a)--(b) node[midway,above right] {};
\draw[->-=.33] (b)--(c) node[midway,above right] {};
\end{tikzpicture}
\caption{$I^3$}
\end{subfigure}
   \caption{The graphs $I^1$, $I^2$, and $I^3$.}\label{fig1}
\end{figure}
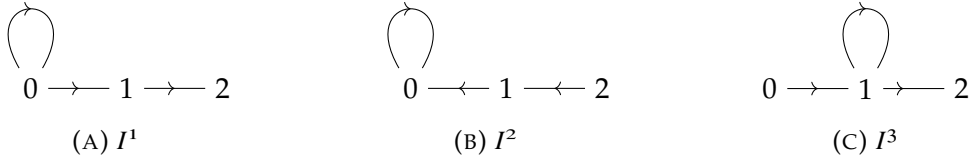

\begin{prop}\label{prop:Mycielskian}
Suppose $G$ is a digraph with $|V(\N(G))|=n>0$. Then we have a 
homotopy equivalence 
\begin{equation*}
    \N(M^1(G))  \simeq \N(G) \sqcup \{\text{pt}\},
\end{equation*}
and an isomorphism of simplicial complexes
\begin{equation*}
        \N(M^2(G)) \cong \N(G) \sqcup \Delta^{n-1}.
\end{equation*}
Here, $\Delta^m$ denotes a simplex of dimension $m$.  In particular, we have $\N(M^1(G)) \simeq \N(M^2(G))$.  In addition we have a homotopy equivalence
\begin{equation*}
 \N(M^3(G))  \simeq S (\N(G)),
 \end{equation*}
 where $S(X)$ denotes the suspension of the topological space $X$.
\end{prop}

\begin{proof}
It is easy to observe that $\N(M^1(G))\cong \{A\times 0 \cup B \times 1 : A\cup B \in \N(G)\} \sqcup \{[(g,2)]\}$. Let $P$ and $Q$ denote the face posets of $\N(M^1(G))\setminus \{[(g,2)]\}$ and $\N(G)$, respectively. Define a poset map $f: P \rightarrow Q$ as $f(A\times 0 \cup B \times 1)=A\cup B$. Clearly, $f^{-1}(q)$ has $q \times 0 \cup q\times 1$ as a maximal element and hence is contractible for any $q\in Q$. Further, if $p=A\times 0 \cup B\times 1\in P$ and $q\in Q$ such that $q \subseteq A\cup B$, then $(A\cap q)\times 0 \cup (B\cap q) \times 1$ is a maximal element of the poset $f^{-1}(q)\cap P_{\leq p}$. Therefore, from Lemma \ref{lem:posetfiber} we get that $\N(M^1(G))  \simeq \N(G) \sqcup \{[(g,2)]\}= \N(G) \sqcup \{pt\}$.

Now note that $\N(M^2(G)) = \{A\times 0 : A \in \N(G)\} \cup \{V(\N(G))\times 1\}$. Since $\{V(\N(G))\times 1\}$ is a simplex on $n$ vertices, we have $\N(M^2(G)) \cong \N(G) \sqcup \Delta^{n-1}$. The fact that $\N(M^3(G))  \simeq S (\N(G))$ is implied by the following observation:
\begin{equation*}
    \begin{split}
        \N(M^3(G)) & = \{(A\times 1) \cup [(g,2)]: A \in \N(G)\} \cup \{\outN_{M^3(G)}([(g,0)])\}   \\
        & = \text{Cone}_{[(g,2)]}(\N(G)\times 1) \cup 
        \{V(\N(G))\times 1\}.
    \end{split}
\end{equation*}
\end{proof}

\subsection{Bipartite graphs and vanishing homology}

    In the case of undirected graphs, results of Lov\'asz \cite{Lov} and Kahle \cite{Kah} imply that if $G$ is a graph not containing a complete bipartite subgraph $K_{a,b}$ for some $a+b = m$, then the neighborhood complex ${\mathcal N}(G)$ admits a strong deformation retract to a complex of dimension at most $m-3$ (and hence cannot have homology in degree larger than $m-3$).  In particular, if the max degree of $G$ is $d$ (so that $G$ is missing a subgraph of the form $K_{1,d+1}$), then ${\mathcal N}(G)$ cannot have homology in degree larger than $d-1$. We prove here that a similar result holds in the case of out-neighborhobd complex of any digraph. We start with a definition.

    \begin{defn} \label{def:bipartite}
       For integers $m,n$ define $\overrightarrow{K}_{m,n}$ to be the \defi{complete bipartite digraph}, the directed graph on vertex set $\{a_1,\dots,a_m,b_1, \dots, b_n\}$ and edge set $\{(a_i,b_j) : i \in [m], ~ j\in [n]\}$.
        \end{defn}

    \begin{prop}\label{prop:bipartite}
    Suppose $G$ is a digraph not containing a copy of $\overrightarrow{K}_{m,n}$ for some $m$ and $n$. Then the complex $\N(G)$ admits a strong deformation retract onto a complex of dimension at most $m+n-3$.
    \end{prop}
    \begin{proof}
    Our argument follows the strategy of Kahle's proof for the undirected case \cite[Lemma 2.3]{Kah}.
    Let $P$ be the face poset of $\N(G)$ and for $X \in P$ let $\outN_G(X)=\bigcap\limits_{x\in X} \outN_G(x)$ and $\inN_G(X)=\bigcap\limits_{x\in X} \inN_G(x)$. 
    Define $\nu : P \rightarrow P$ by $\nu(X)=\outN_G(\inN_G(X))$. 
    Clearly, $\nu$ is an order preserving poset map. Also, it is easy to see that for any $X \in P$ we have $X \subseteq \nu(X)$ and $\nu^2(X)=\nu(X)$. Thus, from \cite[Corollary 10.12]{Bjo95} and the discussion below that, we get that $\nu$ induces a strong homotopy equivalence between ${\mathcal O}(P)$ and ${\mathcal O}(\nu(P))$. 
    
    Now let $r=m+n-3$ and suppose $X_{r+2}\subsetneq X_{r+1}\subsetneq \dots \subsetneq X_{1}$ is a face of ${\mathcal O}(\nu(P))$ of dimension $r+1$. For $i \in [r+2]$, let $Y_i=\inN_G(X_i)$. By definition we have $X_i=\nu(Z_i)$ for some $Z_i\in P$. Thus \[\outN_G(Y_i)=\outN_G(\inN_G(X_i))=\nu(X_i)=\nu^2(Z_i)=\nu(Z_i)=X_i,\]
    \noindent
    implying that $Y_{r+2} \supsetneq Y_{r+1} \supsetneq \dots \supsetneq Y_{1}$. Since $Y_1\neq \emptyset$, we have that $Y_m$ contains at least $m$ vertices of $G$ and $X_m$ contains at least $r+3-m\geq n$ vertices. If $m\leq n$, then $Y_n$ and $X_n$ span a subgraph of $G$ which contains a copy of $\overrightarrow{K}_{m,n}$. Otherwise $Y_m$ and $X_m$ span a subgraph of $G$ which contains a copy of $\overrightarrow{K}_{m,n}$, a contradiction to our assumption on $G$.  We conclude that the dimension of ${\mathcal O}(\nu(P))$ is at most $r$.
    \end{proof}
    
 In the case of undirected graphs, the fact that bipartite subgraphs of a graph $G$ influences the topology of the neighborhood complex ${\mathcal N}(G)$ allowed Kahle \cite{Kah} to study the behavior of random neighborhood complexes.  Our Proposition \ref{prop:bipartite} suggests that one can study $\N(G)$ of random digraphs in a similar way.

 We note that Proposition \ref{prop:bipartite} also of course holds for the in-neighborhood complex $\Nin(G)$.  As an application of these results we obtain a vanishing theorem for homology of neighborhood complexes.  In what follows we say that a digraph $G$ is \defi{simple} if there is at most one directed edge between any pair of vertices (with loops allowed).
 
 \begin{thm}\label{thm:tophomology}
 Suppose $G$ is a simple digraph on at most $m = 2n+2$ vertices.  Then $\tilde{H}_i(\N(G)) = 0$ for all $i \geq n$.
 \end{thm}
 \begin{proof}
First note that since $\binom{m}{2}< m(n+1)$, $G$ has a vertex of outdegree less than $n+1$. Without loss of generality, let $V(G)=[m]$ (if $|V(G)| < m$ we add the appropriate number of isolated vertices) and let $\{1,\dots,t\}$ denote the vertices of outdegree less than $n+1$. This implies that $\outdeg_G(j)>n$, so that $\indeg_G(j)\leq n$ for all $j \in \{t+1,\dots,m\}$.
 
 Let $G_1$ be the digraph obtained from $G$ by removing all out-edges from vertices $1,2\dots,t$. Clearly,
 \[\N(G)=\N(G_1)\cup  \Delta^{\outN_G(1)} \cup \dots \cup  \Delta^{\outN_G(t)}.\]
 Recall that $\Delta^{S}$ denotes a simplex on vertex set $S$.  Since $|\outN_G(j)|\leq n$ for all $j \in [t]$,
 \begin{equation*}
     \tilde{H}_i(\N(G))\cong \tilde{H}_i(\N(G_1)), \hspace*{0.2cm} \forall i \geq n.
 \end{equation*}
 From Proposition \ref{prop:outin}, we know that $\N(G_1)\simeq \Nin(G_1)$. Therefore, 
 \begin{equation}\label{eq:removeoutsmallsets}
     \tilde{H}_i(\N(G))\cong \tilde{H}_i(\Nin(G_1)), \hspace*{0.2cm} \forall i \geq n.
 \end{equation}
 
 Now let $G_2$ be the digraph obtained from $G_1$ by removing all in-edges from vertices $t+1,t+2,\dots,m$. It is again clear that 
 \[\Nin(G_1)=\Nin(G_2)\cup  \Delta^{\inN_{G_1}(t+1)} \cup \dots \cup  \Delta^{\inN_{G_1}(m)}.\]
 Since $\indeg_G(j)\leq n$, we have  $\indeg_{G_1}(j)\leq n$ for all $j \in \{t+1,\dots,m\}$ and hence
 \begin{equation}\label{eq:removeinsmallsets}
     \tilde{H}_i(\Nin(G_1))\cong \tilde{H}_i(\Nin(G_2))\cong \tilde{H}_i(\N(G_2)), \hspace*{0.2cm} \forall i \geq n.
 \end{equation}
 
 From Equations \eqref{eq:removeoutsmallsets} and \eqref{eq:removeinsmallsets}, to prove our result, it is enough to show that $\tilde{H}_i(\N(G_2))=0$ for all $ i \geq n.$ In fact we will show that the complex $\N(G_2)$ admits a strong deformation retract onto a complex of dimension at most $n-1$.
 
 Observe that, in $G_2$, all the edges are from the set $\{t+1,\dots,m\}$ to the set $\{1,\dots,t\}$. Therefore, if $t\leq n$ then the complex $\N(G_2)$ is of dimension at most $n-1$. If $t>n+1$, then $G_2$ does not contain a copy of $\overrightarrow{K}_{n+1,1}$ and therefore the result follows from Proposition \ref{prop:bipartite}. Now let $t=n+1$. Then, either $G_2$ does not contain a copy of $\overrightarrow{K}_{1,n+1}$ or $\N(G_2)$ is a simplex on vertex set $[t]$. In the first case the result again follows from Proposition \ref{prop:bipartite} and the second case is trivial. The result follows.
 \end{proof}
 
 Recall that a simplicial complex $X$ is \defi{$n$-Leray} if $\tilde{H}_i(Y) = 0$ for all induced subcomplexes $Y \subset X$ and $i \geq n$. Equivalently $X$ is $n$-Leray if $\tilde{H}_i(\lk_{X}(\sigma)) = 0$ for all $i \geq n$ and for any face $\sigma \in X$.  As a corollary to Theorem \ref{thm:tophomology} we get the following.
 
\begin{cor}\label{cor:Leray}
Suppose $G$ is a simple digraph on at most $m = 2n+2$ vertices.  Then the complex $\N(G)$ is $n$-Leray.
\end{cor}
\begin{proof}
Without loss of generality assume $V(G) = [m]$ (once again if $|V(G)| < m$ we add sufficiently many isolated vertices).  Let $\sigma$ be a face of $\N(G)$ and $\{x_1,\dots,x_t\}$ be the set of vertices of $G$ such that $\sigma\subseteq \outN_{G}(x_i)$ for all $i \in [t]$. Define a directed graph $G_\sigma$ with $V(G_\sigma)=[m]$ and $ij\in E(G_\sigma)$ if and only if $i\in \{x_1,\dots,x_t\}$ and $j \in \outN_{G}(i)\setminus \sigma$. It is easy to observe that $\lk_{\N(G)}(\sigma)=\N(G_\sigma)$. Using Theorem \ref{thm:tophomology}, we get that $\tilde{H}_i(\lk_{\N(G)}(\sigma)) = 0$ for all $i \geq n$ and hence the result follows.
\end{proof}

\section{Tournaments and reconfiguration}\label{sec:tourn}

We let $K_n$ denote the complete (undirected) graph with vertex set $[n] = \{1,2, \dots, n\}$.  Recall that a \defi{tournament} is some choice of orientations on the edges of $K_n$. In this context we will use $T_n$ to denote an arbitrary tournament.   We let $\tourn$ denote the \defi{standard $n$-tournament}, the directed graph with vertex set $[n]$ and with edges $(i,j)$ for all $i \rightarrow j$ if $i < j$.  Up to isomorphism this is the unique \emph{acyclic} tournament on $n$ vertices, and is also sometimes called the \defi{transitive} tournament.

The study of tournaments has a long history in combinatorics.  Many of the important properties of tournaments were first established by Landau \cite{Lan} in his study of dominance relations in flocks of chickens.  More recently tournaments have found applications in voting theory \cite{Las97} and social choice theory \cite{Brandt16}.  A primary statistic of a tournament is its \emph{outdegree sequence} $(c_1, c_2, \dots, c_n)$, the nonincreasing sequence of outdegrees of the vertices of $T_n$. A well-known theorem of Landau characterizes which sequences arise as outdegree sequences of tournaments.

In this section we study the connectivity and more general topology of the complexes $\Hom(G,T_n)$ for various tournaments $T_n$ and digraphs $G$.   We first consider the case $G=K_2$, where $\Hom(G,T_n)$ is recovered up to homotopy type by $\N(T_n) \simeq \Nin(T_n)$.  The goal here is to understand topological properties of the neighborhood complexes in terms of combinatorial properties of the underlying tournament $T_n$.  

First note that, if $T_n = \tourn$ is transitive, then $\outN_{T_n}(1) =\{2, \dots, n\}$ so that neighborhood complex $\N(\tourn)$ is a simplex, and hence contractible. If $T_n$ is not transitive, it is not so clear what (if any) restriction there might be on the topology of $\N(T_n)$.

In Table \ref{table:tourn} we have listed all isomorphism types of tournaments on 5 vertices, along with their outdegree sequence and the ranks of the homology of their neighborhood complex.   Note that two tournaments can have the same outdegree sequence but have neighborhood complexes with different homotopy types. Also note that from Theorem \ref{thm:tophomology} we cannot have nontrivial homology in rank 2 or above.

\renewcommand{\arraystretch}{1.5}
\begin{table}[H]
\begin{tabular}{|c|c|c|}
\hline
Adjacency (out neighbors) in $T_5$ & Outdegree sequence & $\tilde{H}_i(\N(T_5))$ \\ \hline
1: [], 2: [1], 3: [1, 2], 4: [1, 2, 3], 5: [1, 2, 3, 4] & (4, 3, 2, 1, 0) & {0: 0, 1: 0, 2: 0, 3: 0} \\ \hline
1: [3], 2: [1], 3: [2], 4: [1, 2, 3], 5: [1, 2, 3, 4] &  (4, 3, 1, 1, 1) & {0: 0, 1: 0, 2: 0, 3: 0} \\ \hline
1: [], 2: [1, 4], 3: [1, 2], 4: [1, 3], 5: [1, 2, 3, 4]
 &  (4, 2, 2, 2, 0) & {0: 0, 1: 0, 2: 0, 3: 0} \\ \hline
 1: [4], 2: [1], 3: [1, 2], 4: [2, 3], 5: [1, 2, 3, 4] & (4, 2, 2, 1, 1)  & {0: 0, 1: 0, 2: 0, 3: 0} \\ \hline
2: [], 2: [1], 3: [1, 2, 5], 4: [1, 2, 3], 5: [1, 2, 4] & (3, 3, 3, 1, 0) & {0: 0, 1: 0, 2: 0} \\ \hline

1: [], 2: [1, 4, 5], 3: [1, 2], 4: [1, 3], 5: [1, 3, 4] & (3, 3, 2, 2, 0) & {0: 0, 1: 0, 2: 0}  \\ \hline

1: [4], 2: [1], 3: [1, 2], 4: [2, 3, 5], 5: [1, 2, 3] &  (3, 3, 2, 1, 1) & {0: $\Z$ , 1: 0, 2: 0}  \\ \hline
1: [4], 2: [1], 3: [1, 2, 5], 4: [2, 3], 5: [1, 2, 4] & (3, 3, 2, 1, 1) & {0: 0, 1: 0, 2: 0}   \\ \hline

1: [4, 5], 2: [1], 3: [1, 2], 4: [2, 3], 5: [2, 3, 4] & (3, 2, 2, 2, 1) & {0: 0, 1: 0, 2: 0} \\ \hline

1: [4], 2: [1, 5], 3: [1, 2], 4: [2, 3], 5: [1, 3, 4] & (3, 2, 2, 2, 1)  & {0: 0, 1: $\Z$, 2: 0}  \\ \hline

1: [5], 2: [1, 4], 3: [1, 2], 4: [1, 3], 5: [2, 3, 4] & (3, 2, 2, 2, 1) & {0: $\Z$, 1: $\Z \times \Z$, 2: 0}  \\ \hline

1: [4, 5], 2: [1, 5], 3: [1, 2], 4: [2, 3], 5: [3, 4] & (2, 2, 2, 2, 2)  & {0: 0, 1: $\Z$}  \\ \hline
\end{tabular}
\medskip
\caption{A list of all tournaments $T_5$ on 5 vertices, with their outdegree sequence and homology groups of $\N(T_5)$. These computations were done using Sage \cite{Sage}. }
\label{table:tourn}
\end{table}

One can check that the third to last entry $T_5$ in Table \ref{table:tourn} in fact satisfies $\N(T_5) \simeq {\mathbb S}^1$.  Our next result says that any $n$-sphere, for $n\geq 1$, can be realized (up to homotopy type) as $\N(T_{2n+3})$ for some choice of tournament $T_{2n+3}$. 

\begin{prop}\label{prop:spheres}
Suppose $n\geq 1$ is an integer. Then there exists a tournament $T_{2n+3}$ on $2n+3$ vertices such that $\N(T_{2n+3})$ collapses onto the boundary of a simplex of dimension $n+1$. In particular, $\N(T_{2n+3})\simeq \mathbb{S}^n$. 
\end{prop}
\begin{proof}
We will describe the graphs $T_{2n+3}$ explicitly. For this let $V(T_{2n+3})=[2n+3]=\{1,2,\dots,2n+3\}$. For $n=1$, define the edges of $T_5$ according to out-neighborhoods as follows: $\outN_{T_5}(5)=\{1,3,4\},~\outN_{T_5}(4)=\{2,3\},~\outN_{T_5}(3)=\{1,2\},~\outN_{T_5}(2)=\{1,5\},~\outN_{T_5}(1)=\{4\}$. Observe that $\N(T_5)$ collapses onto the boundary of the $2$-simplex $\{1,2,3\}$ (this can be done by collapsing the free pairs $(\{5\},\{1,5\})$ and $(\{4\},\{1,3,4\})$). 

For $n\geq 2$, we define the edges of $T_{2n+3}$ as follows.
\begin{equation*}
\outN_{T_{2n+3}}(i) = 
\begin{cases}
[i-1]\setminus \{ i-n-2 \}, & \text{ if } i \in \{ n+3,n+4 ,\dots, 2n+3\},\\
[i-1], & \text{ if } i =n+2,\\
[i-1]\sqcup \{n+2+i\}, & \text{ if } i \in \{2,3,\dots, n+1\},\\
\{n+3\}, &  \text{ if } i=1.\\
\end{cases}
\end{equation*}
Recall that we use $\langle \alpha_1,\dots,\alpha_t \rangle$ to denote the simplicial complex with facets $\alpha_1, \dots, \alpha_t$. Observe that in the tournament $T_{2n+3}$ defined above, we have $\outN_{T_{2n+3}}(i)\subset \outN_{T_{2n+3}}(2n+3)$ for each $i \in [n]$. This implies that $\N(T_{2n+3})=\langle \outN_{T_{2n+3}}(i) : i \in \{n+1,\dots, 2n+3\}\rangle$.

Our strategy now will be to collapse each $\outN_{T_{2n+3}}(i)$, for $i \geq n$, down to a distinct $n$-dimensional face with vertices among the set $[n+2]$. Since $(\{2n+3\},\outN_{T_{2n+3}}(n+1))$ is a free pair in $\N(T_{2n+3})$ and $\outN_{T_{2n+3}}(n+1)\setminus\{2n+3\}=[n] \subset \outN_{T_{2n+3}}(2n+3)$, we get that $\N(T_{2n+3})\Searrow X_1 := \langle \outN_{T_{2n+3}}(i) : i \in \{n+2,\dots, 2n+3\}\rangle$.
Next we see that  $(\{3,\dots,n+1,n+3\},\outN_{T_{2n+3}}(n+4))$ is a free pair in $X_1$, which implies that $X_1\Searrow X_2 := \langle \outN_{T_{2n+3}}(i) : i \in \{n+5,\dots, 2n+3\} \rangle \cup \langle [n+1], [n+2]\setminus \{1\}, [n+2]\setminus \{2\} \rangle$. 

For $n>2$, we now collapse the facets $\outN_{T_{2n+3}}(i)$ in $X_2$ for $i \in \{n+5,\dots,2n+2\}$ in order of increasing $i$, using the following sequence of free pairs:
\begin{equation*}
 \begin{split}
 & \big(\{i-n-1,i-n,\dots,n+1,n+3\},\outN_{T_{2n+3}}(i)\big),\\
     & \Big{\{}\big(\{i-n-1,i-n,\dots,n+1,j\},\outN_{T_{2n+3}}(i)\setminus \{n+3,\dots,j-1\}\big): n+4 \leq j \leq i-1  \Big{\}}. 
 \end{split}
\end{equation*}
\noindent
For $i \in \{n+5,\dots,2n+2\}$, after collapsing $\outN_{T_{2n+3}}(i)$, we get that $ \N(T_{2n+3}) \Searrow \langle \outN_{T_{2n+3}}(j) : j \in \{i+1,\dots, 2n+3\} \rangle \cup \langle [n+1]\rangle \cup \langle [n+2]\setminus \{j-n-2\}: j \in \{n+3,\dots,i\} \rangle$. Therefore, 
$$\N(T_{2n+3}) \Searrow \langle \outN_{T_{2n+3}}(2n+3) \rangle \cup \langle [n+1]\rangle \cup \langle [n+2]\setminus \{j-n-2\}: j \in \{n+3,\dots,2n+2\} \rangle.$$

Finally (or if $n=2$), we collapse $\outN_{T_{2n+3}}(2n+3)$ using the following sequence of free pairs:
$$\big(\{n+3\},\outN_{T_{2n+3}}(2n+3)\big),~ \Big{\{}\big(\{j\},\outN_{T_{2n+3}}(2n+3)\setminus \{n+3,\dots,j-1\}\big): n+4 \leq j \leq 2n+2 \Big{\}}.$$
Thus we have
\begin{equation*}
    \begin{split}
        \N(T_{2n+3})  \Searrow & \langle [n+1]\rangle \cup \langle [n+2]\setminus \{j-n-2\}: j \in \{n+3,\dots,2n+3\} \rangle \\
        & = \langle [n+2]\setminus \{i\}: i \in [n+2]\rangle = \partial (\Delta^{[n+2]}). 
    \end{split}
\end{equation*}
\end{proof}

Proposition \ref{prop:spheres} demonstrates that there exist tournaments on $2n+3$ vertices whose out neighborhood complex has nonzero homology in dimension $n$.  Note that by Theorem \ref{thm:tophomology} this is the minimum number of vertices of a digraph that could achieve this, since $\N(T_m)$ is $n$-Leray if $m \leq 2n+2$.

\subsection{Reconfiguration for tournaments}\label{sec:reconfig}

We next consider the notion of \emph{reconfiguration} in the context of homomorphisms of digraphs.  As discussed in Section \ref{sec:intro} the concept of reconfigurations of colorings of undirected graphs has been extensively studied.  Here for an $n$-colorable graph $G$, we let $C_n(G)$ denote the \defi{$n$-color graph of $G$}, by definition the graph with node set given by the proper $n$-colorings of $G$, with two adjacent whenever the corresponding colorings differ on precisely one vertex of $G$.  One can see that $C_n(G)$ is equivalent to the one-skeleton of $\oldHom(G,K_n)$. The question of connectivity of $C_n(G)$ (and hence of $\oldHom(G,K_n)$) is of interest, as it describes how one can move from one coloring of $G$ to another by local moves involving changing the color at one vertex.  As discussed in Section \ref{sec:intro} it also relevant in the context of sampling from colorings of a graph.

In \cite{CHJconnected} Cereceda, van den Heuvel, and Johnson show that if an undirected graph $G$ has chromatic number $\chi(G) = n$ for $n \leq 3$, then $C_n(G)$ is always disconnected.  Other the hand, if $n > 3$ there exist graphs $G$ with chromatic number $n$ for which $C_n(G)$ is connected, and also $n$-chromatic graphs $H$ for which $C_n(H)$ is disconnected. In \cite{CHJmixing} the same group of authors characterize the bipartite graphs $G$ for which $\oldHom(G, K_3)$ is connected. They further show that the problem of deciding the connectedness of $\oldHom(G,K_3)$ for $G$ a bipartite graph is coNP-complete, but that restricted to planar bipartite graphs, the question is answerable in polynomial time.

We now apply these ideas to digraphs, using the $1$-skeleton of the $\Hom$ complex to define our notion of reconfiguration. Although the concept makes sense for homomorphisms $G \rightarrow H$ of any digraphs $G$ and $H$ we will restrict to the case that $H = T_n$ is a tournament.  This is motivated by the notion of $n$-colorings of graphs and digraphs. For instance, if $G$ is a digraph one defines the \defi{oriented chromatic number} $\ch(G)$ as the minimum
$n$ for which there exists a homomorphism $f:G \rightarrow T_n$ for some tournament on $n$ vertices. We refer to \cite{KosSopZhu} for more details regarding oriented chromatic number.  Here we are interested in connectivity properties of $\Hom(G,T_n)$ for various choices of tournaments $T_n$.  Our first result says the situation is very well understood for the case that $T_n = \tourn$ is transitive.

\begin{thm}\label{thm:tourncollapse}
Suppose $G$ is a digraph, and let $\overrightarrow{K_n}$ denote the transitive tournament on $n$ vertices.  Then $\Hom(G,\overrightarrow{K_n})$ is either empty or collapsible.
\end{thm}

\begin{proof}
We prove the statement by induction on $n$. Note that by Proposition \ref{prop:coprod} we can assume that $G$ is connected. The statement is clearly true for $n=1$ and similarly we see that $\Hom(G,\overrightarrow{K}_2)$ is either empty, a vertex (in the case that $G$ has at least one edge and admits a homomorphism $G \rightarrow \overrightarrow{K}_2$), or an edge (in the case that $G$ is an isolated vertex). Not let $n\geq 3$ and suppose that $\Hom(G, \tourn)\neq \emptyset$, which implies that the graph $G$ is acyclic. Let $A=\{a_1,\dots, a_r\}$ denote the set of all vertices of $G$ with outdegree $0$. Since $G$ is acyclic, $A\neq \emptyset$. Now let $S$ be the set of all multihomomorphisms $f\in \Hom(G, \tourn)$ such that $n \in f(a_1)$ and $|f(a_1)|>1$. 

If $S=\emptyset$, then it is easy to see that every $f \in \Hom(G,\tourn)$ satisfies $f(a_1) = n$.   If $S \neq \emptyset$, we will collapse $\Hom(G,\tourn)$ onto the complex $\{f\in \Hom(G,\tourn) : f(a_1)=n\}.$ Assuming $S\neq \emptyset$, define a matching
$\mu:S \rightarrow \Hom(G,\tourn)\setminus S$ as follows:
$$\mu(f)(v)=\begin{cases}
f(v) & {\rm{if~}} v\neq a_1,\\
f(v)\setminus\{n\} & {\rm{if~}} v=a_1.
\end{cases}$$
Observe that $\mu(f)\in \Hom(G,\tourn)$ for each $f\in S$ and $\mu(f)\prec f$. We now show that $S$ (with the map $\mu$ defined above) is an acyclic partial matching. For a contradiction, suppose $S$ is not an acyclic matching and let $\mu(f_1) \prec f_2\succ \mu(f_2) \prec f_3\succ \dots\succ \mu(f_t) \prec f_1$ be a minimal cycle. Since $\mu(f_t) \prec f_1$ and $n \notin \mu(f_t)(a_1)$, we conclude that $f_t=f_1$, which is a contradiction to the minimality of the cycle. Hence, from Proposition \ref{prop:dmtmain} we have
\[\Hom(G,\tourn)\Searrow \Hom(G,\tourn)\setminus (S\sqcup \mu(S))= \{f\in \Hom(G,\tourn) : f(a_1)=n\}.\]

Furthermore, for any $i\in [r-1]$, using similar arguments as above we can show that
\[\{f\in \Hom(G,\tourn) : f(a_j)=n\text{ for each } j \in [i]\} \Searrow \{f\in \Hom(G,\tourn) : f(a_j)=n\text{ for each } j \in [i+1]\}.\]

Hence $\Hom(G,\tourn)\Searrow \{f \in \Hom(G,\tourn) : f(a_j)=n\text{ for each } j \in [r]\} \cong\Hom(G\setminus A, \overrightarrow{K}_{n-1})$. The result then follows from induction. 
\end{proof}

\begin{rem}
One may wonder if there exists other acyclic graphs $G$ with the property that $\Hom(T,G)$ is empty or collapsible for every digraph $T$.  We observe that not every acyclic graph $G$ has this property, since in the proof of Proposition \ref{prop:universal} we see that any homotopy type can be realized as $\Hom(K_2, G)$ for some acyclic (and in fact transitive) graph $G$. Another natural question to ask is whether there exists a graph $G$ with the property that $\Hom(T,G)$ is \emph{nonempty} and collapsible for any $T$.  First note that such a graph must contain a looped vertex since in particular it must admit a graph homomorphism from any graph $T$.  We address this question in Section \ref{sec:homotopy}.
\end{rem}

 Theorem \ref{thm:tourncollapse} in particular implies that the one-skeleton $(\Hom(G,\tourn))^{(1)}$ is connected for any $G$, assuming some homomorphism $G \rightarrow \tourn$ exists.  Note that a digraph $G$ is acyclic if and only if $G \rightarrow \tourm$ for some $m$.  In this case any homomorphism $G \rightarrow \tourm$ can be `reconfigured' into any other.  In this context one is interested in the diameter of the configuration space, and here we can give a tight bound.

\begin{thm}\label{thm:homgconnected}
Suppose $G$ is a digraph, and let $\overrightarrow{K_n}$ denote the transitive tournament on $n$ vertices.  Then the one-skeleton $(\Hom(G,\overrightarrow{K_n}))^{(1)}$ is either empty or has diameter at most $|V(G)|$.
\end{thm}

\begin{proof}
Without loss of generality assume $\Hom(G,
\tourn)$ is nonempty. From Theorem \ref{thm:tourncollapse} we know that $\Hom(G,\tourn)$ is connected, and here we will explicitly construct a path between any two elements of length bounded by $|V(G)|$. For this let $f,g \in \Hom_0(G, \tourn)$ such that $f\neq g$. Define a function $h:G\rightarrow \tourn$ by
$$h(v)=\mathrm{min}\{f(v),g(v)\}, ~\forall~ v \in V(G).$$
\begin{claim}\label{claim:HomGconnected} For any digraph $G$ such that $\Hom(G,
\tourn)$ is nonempty, we have the following.
\begin{enumerate}
    \item $h\in \Hom_0 (G,\tourn)$, and
    \item $f$ and $g$ are connected to $h$ in $\Hom(G,\tourn)$.
\end{enumerate}
\end{claim}
To establish Claim \ref{claim:HomGconnected}(1), suppose $h$ is not a homomorphism, so that there exists an edge $(u,v)\in E(G)$ such that $h(v)\leq h(u)$. Since $f$ and $g$ are homomorphisms, we have $f(u)<f(v)$ and $g(u)<g(v)$ . Without loss of generality suppose $f(u) \leq g(u)$, so that $h(u)=f(u)$ (by definition of $h$). Since $h(v)\leq h(u)=f(u)<f(v)$, in particular we have that $h(v) \neq f(v)$ and hence $h(v)=g(v)$. This implies that $g(u)<g(v)=h(v)\leq h(u)=f(u)$, which contradicts the fact that $f(u)\le g(u)$. A similar argument works when $h(u)=g(u)$. Therefore $h:G \rightarrow \tourn$ is a homomorphism.

We now show that $f$ is connected to $h$ in $\Hom(G,\tourn)$. Let $A=\{v\in V(G): h(v)\neq f(v)\}$. Clearly, $g(v)<f(v)$ for all $v\in A$. Let $A=\{a_1,a_2,\dots,a_t\}$ with $g(a_1)\leq g(a_2)\leq\dots \leq g(a_t)$. For $i\in [t]$, define  $\sigma_i: V(G)\rightarrow 2^{[n]}\setminus \emptyset$ as 
$$\sigma_i(x)=
\begin{cases}
g(x) & \text{if} ~x \in \{a_1,\dots,a_{i-1}\},\\
\{f(x),g(x)\} & \text{if} ~x=a_i,\\
f(x)& \text{otherwise}.
\end{cases}$$

We claim that the $\sigma_i$ are edges in the complex $\Hom(G,\tourn)$ that give a path from vertex $f$ to vertex $h$. The fact that they are $1$-dimensional and give the desired path is clear, and we are left to verify that $\sigma_i$ is a multihomomorphism for each $i \in [t]$. For this, we have to show that for any $(u,v)\in E(G)$, $\sigma_i(u)\times \sigma_i (v)\subseteq E(\tourn)$. We split the proof in four cases. 

\begin{enumerate}
    \item[(i)] Let $\{u,v\}\cap \{a_1,\dots,a_i\}=\emptyset$. In this case $\sigma_i(u)\times \sigma_i (v)\subseteq E(\tourn)$ comes from the fact that $f \in \Hom(G,\tourn)$.
    
    \item[(ii)] Let $u=a_{j_1}$ and $v =a_{j_2}$ for some $j_1,j_2\in [i]$. Since $g(a_1)\leq g(a_2)\leq\dots \leq g(a_i)<f(a_i)$ and $g \in \Hom(G,\tourn)$, we get that $j_1<j_2$ and $\sigma_i(a_{j_1})\times \sigma_i(a_{j_2})\subseteq E(\tourn)$.
    
    \item[(iii)] Let $\{u,v\}\cap \{a_1,\dots,a_i\}=\{u\}$. Since $\sigma_i(v)=f(v)$, the condition $g(u)<f(u)$, along with the fact that $f, g \in \Hom(G,\tourn)$, implies that $\sigma_i(u)\times \sigma_i(v)\subseteq E(\tourn)$. 
    
    \item[(iii)] Let $\{u,v\}\cap \{a_1,\dots,a_i\}=\{v\}$. Since $g(u) < g(v)$ and $u \notin \{a_1,\dots,a_i\}$, we get that $f(u)\leq g(u)$. Thus the result follows from the facts that $g\in \Hom(G,\tourn)$ and $g(v) < f(v)$ .
\end{enumerate}
Observe that the path from $f$ to $h$ obtained above is of length $|A|=|\{v\in V(G):h(v)\neq f(v)\}|$. A similar path of length $|\{v\in V(G):h(v)\neq g(v)\}|$ can be constructed from $g$ to $h$. Hence we get a path from $f$ to $g$ of length $\{v\in V(G):h(v)\neq f(v)\}+\{v\in V(G):h(v)\neq g(v)\}=\{v\in V(G):f(v)\neq g(v)\}\leq|V(G)|$. This establishes the second part of Theorem \ref{thm:homgconnected}.
\end{proof}

We note that Theorem \ref{thm:homgconnected} is tight since for instance the one-skeleton of $\Hom(\tourm, \tourn)$ has diameter $m$ whenever $m < n$.   See Section \ref{sec:mixed} for more details regarding these complexes.

Now suppose $G$ is a connected digraph with oriented chromatic number $\chi_o(G) = 3$, so that $G$ admits a homomorphism to $\overrightarrow{K_3}$ or to the $3$-cycle $\overrightarrow{C}_3$.  Note that for any digraph $G$ (regardless of chromatic number), $\Hom(G,\overrightarrow{C}_3)$ is either empty or has at least three connected components, each of which is a point.  To see this note that $\overrightarrow{C}_3$ has no vertex $v$ satisfying $\indeg(v) \geq 2$ or $\outdeg(v) \geq 2$. In addition, if $f_1 \in \Hom(G,\overrightarrow{C}_3)$ we obtain two other elements $f_2$ and $f_3$ by rotating the images of the vertices of $G$.   Hence with Theorem \ref{thm:homgconnected} this gives a complete picture of mixing of digraphs with directed chromatic number 3. Similarly, if $G$ is connected and satisfies $\chi_o(G) =2$, then we have a complete understanding of $\Hom(G,T_3)$ for any choice of $T_3$.  Namely, we have that $\Hom(G,\overrightarrow{K_3})$ is nonempty and connected (in fact contractible)  whereas $\Hom(G,\overrightarrow{C}_3)$ is nonempty and has at least three components, each of which is a point. With these observations we obtain digraph analogues of the results in \cite{CHJconnected} and \cite{CHJmixing}.

In general, if $G$ has oriented chromatic $n$ for $n \geq 4$ the connectivity of $\Hom(G,T_n)$ will depend on the tournament $T_n$.  Note that for a fixed tournament $T_n$, the complex $\Hom(T_n,T_n)$ is nonempty and has no edges.  Hence $\Hom(T_n,T_n)$ will be connected if and only if $\Aut(T_n)$, the automorphism group of $T_n$, is trivial.

\subsection{Transitive tournaments and mixed subdivisions}\label{sec:mixed}
In this section we further study the special case of homomorphism complexes between transitive tournaments. From Theorem \ref{thm:tourncollapse} we know that $\Hom(\tourm, \tourn)$ is contractible whenever $m \leq n$, but here we employ some notions from geometric combinatorics to precisely describe its structure as a polyhedral complex.

We first review some notions from the theory of mixed subdivisions.  Let $P_1, \dots P_k \subset \mathbb{R}^d$ be convex polytopes which we assume to be full-dimensional.  The \defi{Minokowski sum} is defined as the set
\[P_1 + \cdots + P_k := \{x_1 + \cdots + x_k: x_i \in P_i\}.\]

A \defi{Minkowski cell} of the Minkowski sum $\sum_{i=1}^k P_i$ is any full-dimensional polytope $C = \sum_{i=1} C_i$, where $C_i \subset P_i$ is some (not necessarily full dimensional) polytope whose vertices are a subset of the vertices of $P_i$. A \defi{mixed subdivision} of $P = \sum_{i=1}^k P_i$ is a collection ${\mathcal S}$ of Minkowski cells whose set-theoretical union is $P$, and whose intersections are given as Minkowski sums, so that if $C = \sum_{1=1}^k C_i$ and $C^\prime = \sum_{i=1}^k C_i^\prime$ are Minkowski cells then $C_i \cap C_i^\prime$ is a face of $C_i$ and of $C_i^\prime$.

Now let $\Delta^{b}$ denote the $b$-dimensional simplex given by $\conv\{{\bf e}_0, {\bf e}_1, \dots {\bf e}_b\} \subset {\mathbb R}^b$, where ${\bf e}_i$ for $i = 1, \dots, b$ are the standard basis vectors in ${\mathbb R}^b$ and ${\bf e}_0 = {\bf 0}$ is the origin. If each $P_i = \Delta^i$, then the Minkowski sum $\sum_{i=1}^a P_i$ is simply the dilation $a \Delta^b$ of a simplex, and any Minkowski cell can be thought of as an ordered set $(S_1, \dots, S_a)$, where $S_i \subset \{0,1, \dots,a\}$.  The study of mixed subdivisions of $a \Delta^b$ (which correspond to triangulations of the product of simplices $\Delta^{a-1} \times \Delta^b$ under the \emph{Cayley trick}) is an active area of research, with connections to flag arrangements \cite{ArdBil}, the geometry of products of minors of a matrix \cite{BabBil},  tropical geometry \cite{DevStu}, and determinantal ideals \cite{Stu}.

There is a certain mixed suddivision of $a \Delta_b$ that has a particularly nice combinatorial description. For this let $K_{a,b+1}$ denote the (undirected) complete bipartite graph on vertex sets $[a] \sqcup [b+1]$.  A spanning tree of $K_{a,b+1}$ then defines a Minkowski cell $(S_1, \dots, S_a)$ of $a \Delta^b$, where each $S_i$ is given by the neighbors of the vertex $i = 1,\dots, a$. Embed the vertices of $K_{a,b+1}$ in the plane using coordinates $(0,1), (0,2), \dots (0,a)$, $(1,1), (1,2), \dots, (1,b+1)$.  A spanning tree of $K_{a,b+1}$ is \defi{non-crossing} if its straight line drawing in this embedding does not have crossings.  The collection of all Minkowski cells that correspond to non-crossing spanning trees of $K_{a,b+1}$ defines a mixed subdivision ${\mathcal T}_b^a$ of $a \Delta^b$ that we call the \defi{staircase subdivision}.  We refer to \cite{DelRamSan} for more details. 

\begin{prop}\label{prop:mixedsub}
For any integers $2 \leq m \leq n$, the complex $\Hom(\tourm, \tourn)$ is homeomorphic to ${\mathcal T}_{n-m}^m$.  In particular, $|\Hom(\tourm,\tourn)|$ is homeomorphic to a ball of dimension $n-m$.
\end{prop}

\begin{proof}
The proof of this statement is implicit in \cite{DocEng} but we spell out the details here.  Recall that the maximal cells in the subdivision  ${\mathcal T}_{n-m}^m$ correspond to noncrossing spanning trees of the bipartite graph $K_{m,n-m+1}$, where as above $K_{m,n-m+1}$ has vertex set $[m] \sqcup [n-m+1]$. Each spanning tree $T$ corresponds to a mixed cell given by
\[\Delta_T = \Delta^{N(1)} + \Delta^{N(2)} + \cdots + \Delta^{N(m)},\]
\noindent
where for all $i = 1, \dots m$, $N(i) = N_{K_{m,n-m+1}}(i)$ is the set of vertices adjacent to $i$, and for $S \subset [n]$ we use $\Delta^S$ to denote the simplex determined by vertices in $S$. 

Since $T$ is noncrossing, we have that if $x \in S_i$ and $y \in S_j$ with $i<j$, then $x \leq y$. Hence, if we define $S(1) = N(1)$, $S(2) = \{x+1:x \in N(2)\}$, \dots, $S(m) = \{x+m-1: x \in N(m)\}$ we get an ordered collection $(S(1), S(2), \dots S(m))$ of subsets of $[n]$ with the property that if $x \in S_i$ and $y \in S_j$ with $i<j$ then $x<y$.  These are precisely the maximal cells (corresponding to multihomomorphisms) in $\Hom(\tourm, \tourn)$. The result follows.
\end{proof}

\begin{example}
For examples of the complexes ${\mathcal T}^m_{n-m}$ we refer to Figure \ref{fig:example1}, where the cases of $(m,n)=(2,4)$ and $(m,n) = (3,5)$ are depicted.
\end{example}

\begin{rem}
We note that there exist contractible $\Hom(T_m,T_n)$ that do not arise from transitive tournaments.  For instance, the tournament $T_4$ in Figure \ref{fig:example_hom_contractible} has the cycle 132 but we see that $\Hom(T_2,T_4)$ is contracitble. 
\end{rem}

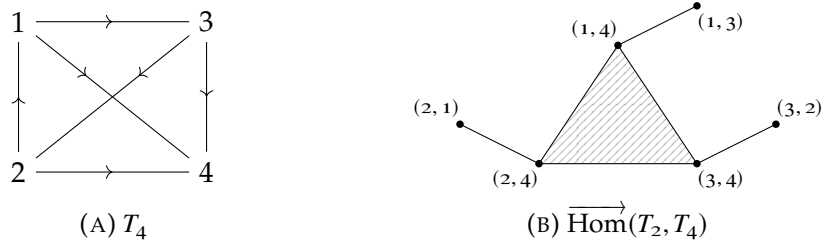
\begin{figure}[H]
\begin{subfigure}[]{0.4\textwidth}
    \centering
    \begin{tikzpicture}
\node (a) at (0,0) {1};
\node (c) at (0,-2) {2};
\node (b) at (2.5,0) {3};
\node (d) at (2.5,-2) {4};
\draw[->-=.5] (c)--(a) node[midway,above right] {};
\draw[->-=.5] (a)--(b) node[midway,above right] {};
\draw[->-=.33] (b)--(c) node[midway,above right] {};
\draw[->-=.5] (b)--(d) node[midway,above right] {};
\draw[->-=.5] (c)--(d) node[midway,above right] {};
\draw[->-=.33] (a)--(d) node[midway,above right] {};
\end{tikzpicture}
\caption{$T_4$}
\end{subfigure}
\begin{subfigure}[]{0.4\textwidth}
    \centering
    \begin{tikzpicture}[scale=1.05]
   \pattern[pattern=north east lines,opacity=0.5] (0.0,  0.0)--(-1.0, -1.5)--(1.0, -1.5)--cycle;
   \path[draw] ( 0.0,  0.0)--(-1.0, -1.5);
   \path[draw] (-1.0, -1.5)--( 1.0, -1.5);
   \path[draw] ( 0.0,  0.0)--( 1.0, -1.5);
   \path[draw] (-1.0, -1.5)--( -2.0, -1.0);
   \path[draw] ( 0.0,  0.0)--( 1.0, 0.5);
   \path[draw] ( 1.0, -1.5)--( 2, -1.0);
   
   \node[dot] at ( 0.0,  0.0) {};
   \node[dot]at (-1.0, -1.5) {};
   \node[dot] at ( 1.0, -1.5) {};
   \node[dot] at ( -2.0, -1.0) {};
   \node[dot] at ( 1.0, 0.5) {};
   \node[dot] at ( 2, -1.0) {};
   
   \node[inner sep=0pt] () at ( -0.3,  0.2) {{\tiny$(1,4)$}};
   \node[inner sep=0pt] () at ( -1.3,  -1.7) {{\tiny$(2,4)$}};
   \node[inner sep=0pt] () at ( 1.3,  -1.7) {{\tiny$(3,4)$}};
   \node[inner sep=0pt] () at ( -2.3,  -0.8) {{\tiny$(2,1)$}};
   \node[inner sep=0pt] () at ( 1.3,  0.3) {{\tiny$(1,3)$}};
   \node[inner sep=0pt] () at ( 2.3,  -0.8) {{\tiny$(3,2)$}};
   \end{tikzpicture}
    \caption{$\Hom(K_2, T_4)$}
\end{subfigure}

   \caption{Non-transitive $4$-tournament $T_4$ such that $\Hom(T_2, T_4)$ is contractible}\label{fig:example_hom_contractible}
\end{figure}

\section{Homotopy of directed graphs}\label{sec:homotopy}

In this section we use the tools developed above to construct various notions of \emph{homotopy} for graph homomorphisms of digraphs. Again this runs parallel to the theory for undirected graphs as worked out in \cite{DocHom}, where a notion of $\times$-homotopy was introduced.  The basic idea is to use the 1-skeleton of the relevant homomorphism complex to define a notion of homotopy between graph homomorphisms.  In the undirected case we have that the complex $\oldHom(G,H)$ is homotopy equivalent to the clique complex on the looped vertices of the graph $H^G$, and hence one can equally well work with $(H^G)^{{o}}$ when considering homotopy. This idea was further explored by the first author and Schultz in \cite{DocSch}, as well as by Chih and Scull in \cite{ChiScu}.

As we will see, a fundamental difference in the directed setting is that there are a number of natural notions of paths in the exponential graph $H^G$, only one of which corresponds to paths in the corresponding homomorphism complex. In particular, if one considers the digraph structure of $H^G$, a notion of direction can be incorporated into the theory. In recent years the concept of a \emph{directed} homotopy theory for discrete structures has found applications in topological data science and the study of concurrency, see for instance \cite{FajGouRau}.  The category of digraphs provides a context where a notion of directed homotopy naturally arises, and seems to be a useful testing ground for such ideas. 

 To motivate our definitions we recall some analogous constructions for undirected graphs discussed in \cite{DocHom}. Here two graph homomorphisms $f,g:G \rightarrow H$ are \defi{$\times$-homotopic} if there exists a path from $f$ to $g$ in the (undirected) homomorphism complex $\oldHom(G,H)$, or equivalently a path along looped vertices in the graph $H^G$.  In \cite{DocHom} this was seen to be equivalent to existence of a graph homomorphism (or \emph{homotopy}) $F: G \times \overleftrightarrow{I}_n \rightarrow H$ with the property that $F(x,0) = f(x)$ and $F(x,n) = g(x)$.  Here $\overleftrightarrow{I}_n$ is a looped path of length $n$ (which we will also think of as a digraph, see below). The $\times$ notation is used to emphasize that the underlying product involved in the internal hom $H^G$ is the \emph{categorical} product. This distinguishes the construction from other approaches (for instance \cite{BBDL}) where the \emph{cartesian} product was used.

We next explore these constructions for digraphs. In this context there are a number of ways one may wish to define the notion of homotopy, based on paths in various parametrization spaces and different notions of the interval object $I_n$ for digraphs. For our first notion of homotopy we consider (usual, topological) paths in the homomorphism complex.

\begin{defn}
Suppose $f,g:G \rightarrow H$ are homomorphisms of digraphs.  We will say that $f$ and $g$ are \defi{bihomotopic}, denoted $f \bhom g$, if there exists a path from $f$ to $g$ in the complex $\Hom(G,H)$.
\end{defn}

In Remark \ref{rem:clique} we discussed the homotopy equivalence $\Hom(G,H) \simeq X({\mathcal G}(H^G))$.  Recall that the latter complex is the clique complex of the (undirected) graph ${\mathcal G}(H^G))$, whose vertices are the looped vertices of the digraph $H^G$, and where $f \sim g$ if both edges $(f,g)$ and $(g,f)$ exist. For any non-negative integer $n$, we define the graph $\overleftrightarrow{I}_n$ to be the digraph with vertex set $\{0,1,\dots n\}$ and edges $(i,j)$ whenever $|i-j| \leq 1$ (so that $\overleftrightarrow{I}_n$ is a bidirected path with loops on every vertex, see Figure \ref{fig:intervals}).  We then have the following observation.

\begin{prop}
Suppose $f,g: G \rightarrow H$ are homomorphisms of digraphs. Then $f$ and $g$ are bihomotopic if and only if there exists a non-negative integer $n$ and a graph homomorphism $F:G \times \overleftrightarrow{I}_n \rightarrow H$ such that $F(x,0) = f(x)$ and $F(x,n) = g(x)$.
\end{prop}

The notion of bihomotopy leads to an equivalence relation on the set of digraphs in the usual way.  Namely, a homomorphism of digraphs $f:G \rightarrow H$ is a \defi{bihomotopy equivalence} if there exists a homomorphism $g:H \rightarrow G$ such that $g \circ f \bhom id_G$ and $f \circ g \bhom id_H$.  In \cite{MatMor} Matushita considers a notion of \emph{strong homotopy} of $r$-sets, extending the notion of $\times$-homotopy of undirected graphs developed by the first author in \cite{DocHom}.  Recall that a $2$-set recovers the notion of a directed graph, and one can see that strong homotopy equivalence for $2$-sets is equivalent to our notion of bihomotopy equivalence. From \cite{MatMor} we then get the  following characterization of bihomotopy equivalence.  

\begin{prop}\cite[Corollary 5.2]{MatMor}\label{prop:equiv}
Suppose $f:G \rightarrow H$ is a homomorphism of digraphs.  Then the following are equivalent.
\begin{enumerate}
    \item The homomorphism $f$ is a bihomotopy equivalence.
    \item For any digraph $T$, the poset map $f_*:\Hom(T,G) \rightarrow \Hom(T,H)$ is a strong homotopy equivalence.
    \item For any digraph $S$, the poset map $f^*:\Hom(H,S) \rightarrow \Hom(G,S)$ is a strong homotopy equivalence.
\end{enumerate}
\end{prop}

For the case of $\times$-homotopy of undirected graphs, foldings were essential in defining certain distinguished elements of a $\times$-homotopy class. We will see that something similar holds for bihomotopy of digraphs.  From Proposition \ref{prop:folding} we have the following.

\begin{lemma}\label{lem:dismant}
Suppose $G$ is a digraph and $v \in G$ is dismantlable. Then the inclusion $G\backslash \{v\} \rightarrow G$ is a bihomotopy equivalence.
\end{lemma}

A digraph is said to be \defi{stiff} if does not contain any dismantlable vertex.  As a corollary to the above results we get the following, see also Corollary 5.10 from \cite{Mat} for the analogous statement for $r$-sets.

\begin{cor}\label{cor:stiff}
Suppose $G$ and $H$ are both stiff digraphs.  Then a homomorphism $f: G \rightarrow H$ is a bihomotopy equivalence if and only if it is an isomorphism.
\end{cor}

As a consequence of Corollary \ref{cor:stiff} each bihomotopy class $[G]$ contains a unique (up to isomorphism) stiff representative.  This representative is a single looped vertex ${\bf 1}$ exactly when $G$ is dismantlable, and in this case we have that $\Hom(T,G) \simeq \Hom(T,{\bf 1})$ is collapsible for any graph $T$.  In \cite{BriWinGibbs} Brightwell and Winkler prove that an undirected graph $G$ is \emph{dismantlable} if and only if $\oldHom(T,G)$ is nonempty and connected for any graph $T$. We have something similar in the directed setting.

\begin{thm}\label{thm:dismant}
Suppose $G$ is a digraph.  Then the following are equivalent.

\begin{enumerate}
    \item $G$ is dismantlable (in the directed sense of Definition \ref{def:folding});
    \smallskip
    \item The stiff representative of $[G]$ is a single looped vertex ${\bf 1}$; 
      \smallskip
    \item
    The complex $\Hom(T,G)$ is nonempty and contractible for any graph $T$;
      \smallskip
    \item 
    The complex $\Hom(T,G)$ is nonempty and connected for any graph $T$.
    \end{enumerate}
\end{thm}

\begin{proof}
The implication $(1) \Rightarrow (2)$ follows from Lemma \ref{lem:dismant}, $(2) \Rightarrow (3)$ follows from Proposition \ref{prop:folding}, and $(3) \Rightarrow (4)$ is immediate.   We prove $(4) \Rightarrow (1)$ by induction on $n = |V(G)|$.    Note that $G$ must have a looped vertex since in particular $\Hom({\bf 1}, G)$ is nonempty. If $G = {\bf 1}$ we are done, and otherwise we have that the constant homomorphism $c: G \rightarrow {\bf 1}$ is distinct from the identity $i:G \rightarrow G$. By assumption we have that $\Hom(G,G)$ is connected and hence there exists a path from $c$ to $i$ in the complex $\Hom(G,G)$. But note that if $j: G \rightarrow G$ is any homomorphism that is adjacent to $i$ in $\Hom(G,G)$ we have some vertex $v \in G$ where $j(v) \neq v$ and where $j(w) = w$ for all $w \neq v$. Hence $G$ admits a folding $G \rightarrow G \backslash \{v\}$.  By our assumption and Proposition \ref{prop:folding} we have that $\Hom(T,G \backslash \{v\})$ is connected for any $T$, and by induction we conclude that $G \backslash \{v\}$ is dismantlable. The result follows.
\end{proof}

\subsection{Other notions of homotopy}

We next consider other constructions of homotopy of digraphs, based on different notions of paths in the exponential digraph $H^G$.  For the first we take advantage of the fact that $H^G$ is itself a directed graph.  Recall that for a digraph $T$ we use $T^{{o}}$ to denote the subgraph of $T$ induced on the looped vertices.  We note that the vertices of $(H^G)^{{o}}$ correspond to the graph homomorphisms $G \rightarrow H$.

\begin{defn}
Suppose $f,g:G \rightarrow H$ are homomorphisms of digraphs.  We say that $f$ and $g$ are \defi{dihomotopic}, denoted $f \dhom g$, if there exists a directed path from $f$ to $g$ in the graph $(H^G)^{{o}}$.
\end{defn}

We note that $\dhom$ is reflexive and transitive, but does not define an equivalence relation since it is not symmetric.  Also note that if we let $\overrightarrow{I}_n$ denote the directed path on $n+1$ looped vertices, then $f \dhom g$ if and only if there exists an integer $n \geq 1$ and a graph homomorphism $F: G \times \overrightarrow{I}_n \rightarrow H$ with $F(x,0) = f(x)$ and $F(x,n) = g(x)$.

Dihomotopy has some perhaps unexpected properties.  For example, if  $f_i, f_j:{\bf 1} \rightarrow \overrightarrow{I}_n$ are the inclusions of of a looped vertex ${\bf 1}$ given by ${\bf 1} \mapsto i$, ${\bf 1} \mapsto j$, then $f_i \dhom f_j$ if and only if $i \leq j$.  Similarly, it is not entirely clear how one should define `dihomotopy equivalence' between two digraphs.  

For our third (and final) notion of homotopy, we allow \emph{any} path in the underlying undirected graph $(H^G)^{{o}}$.

\begin{defn}
Suppose $f,g:G \rightarrow H$ are homomorphisms of digraphs. Then $f$ and $g$ are \defi{line-homotopic}, denoted $f \lhom g$, if there exists a path from $f$ to $g$ in the underlying undirected graph of $(H^G)^{{o}}$.
\end{defn}

The existence of a homotopy $f \lhom g$ is again equivalent to a digraph map $F:G \times I_n \rightarrow H$ with the desired projections.  In this case $I_n$ is a path on $n+1$ looped vertices $\{0,1, \dots, n\}$ where each edge appears in one of \emph{either} direction (so that for all $i$ we have exactly one of $(i,i+1)$ or $(i+1,i)$ is an edge ).   This approach to intervals is also taken in \cite{GLMY}, where the resulting notion of homotopy is based on the \emph{Cartesian product} of digraphs.  In \cite{GLMY} the graph $I_n$ is called a \defi{line digraph} and the set of all line digraphs on $n+1$ vertices is denoted by ${\mathcal I}_n$.  We refer to figure \ref{fig:intervals} for an illustration of the various notions of interval graphs discussed so far.

The homotopy theory of digraphs developed in \cite{GLMY} was inspired by a \emph{homology} theory of digraphs introduced by the authors in previous papers.  In \cite{GLMY} it was shown that these theories interact in the expected ways, including invariance of homology under homotopy equivalence. In addition, the theory extends the $A$-homotopy of (undirected) graphs studied in \cite{BBDL} to the setting of digraphs. In a similar way the line-homotopy theory of digraphs described here can be seen as a generalization of the $\times$-homotopy of \cite{DocHom} to the setting of digraphs.  We do not know if there is a corresponding homology theory in our context, see Section \ref{sec:higherhom} for more discussion.

    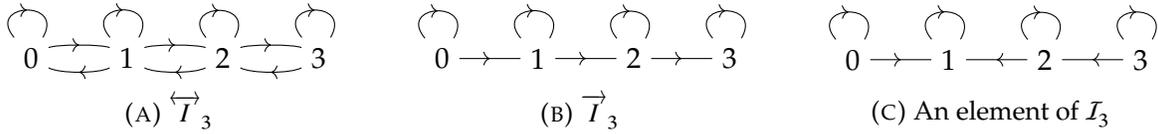
\begin{figure}[H]
\begin{subfigure}[]{0.325\textwidth}
    \centering
    \begin{tikzpicture}[scale=0.85]
\node (a) at (0,0) {0};
\node (b) at (1.5,0) {1};
\node (c) at (3,0) {2};
\node (d) at (4.5,0) {3};

\draw[->-=.5] (a) to [out=30,in=150,looseness=0.3] (b);
\draw[->-=.5] (b) to [out=210,in=330,looseness=0.5] (a);
\draw[->-=.5] (b) to [out=30,in=150,looseness=0.3] (c);
\draw[->-=.5] (c) to [out=210,in=330,looseness=0.5] (b);
\draw[->-=.5] (c) to [out=30,in=150,looseness=0.3] (d);
\draw[->-=.5] (d) to [out=210,in=330,looseness=0.5] (c);

\draw[->-=.5] (a) to [out=130,in=60,looseness=4] (a);
\draw[->-=.5] (b) to [out=130,in=60,looseness=4] (b);
\draw[->-=.5] (c) to [out=130,in=60,looseness=4] (c);
\draw[->-=.5] (d) to [out=130,in=60,looseness=4] (d);
\end{tikzpicture}
\caption{$\overleftrightarrow{I}_3$}
\end{subfigure}
\begin{subfigure}[]{0.325\textwidth}
    \centering
    \begin{tikzpicture}[scale=0.85]
\node (a) at (0,0) {0};
\node (b) at (1.5,0) {1};
\node (c) at (3,0) {2};
\node (d) at (4.5,0) {3};

\draw[->- =.5] (a)--(b) node[midway,above right] {};
\draw[->-=.5] (b)--(c) node[midway,above right] {};
\draw[->-=.5] (c)--(d) node[midway,above right] {};

\draw[->-=.5] (a) to [out=130,in=60,looseness=4] (a);
\draw[->-=.5] (b) to [out=130,in=60,looseness=4] (b);
\draw[->-=.5] (c) to [out=130,in=60,looseness=4] (c);
\draw[->-=.5] (d) to [out=130,in=60,looseness=4] (d);
\end{tikzpicture}
\caption{$\overrightarrow{I}_3$}
\end{subfigure}
\begin{subfigure}[]{0.325\textwidth}
    \centering
    \begin{tikzpicture}[scale=0.85]
\node (a) at (0,0) {0};
\node (b) at (1.5,0) {1};
\node (c) at (3,0) {2};
\node (d) at (4.5,0) {3};

\draw[->- =.5] (a)--(b) node[midway,above right] {};
\draw[->- =.5] (c)--(b) node[midway,above right] {};
\draw[->- =.5] (d)--(c) node[midway,above right] {};

\draw[->-=.5] (a) to [out=130,in=60,looseness=4] (a);
\draw[->-=.5] (b) to [out=130,in=60,looseness=4] (b);
\draw[->-=.5] (c) to [out=130,in=60,looseness=4] (c);
\draw[->-=.5] (d) to [out=130,in=60,looseness=4] (d);
\end{tikzpicture}
\caption{An element of ${\mathcal I}_3$}
\end{subfigure}
\caption{The various interval objects involved in homotopy of digraphs.}\label{fig:intervals}
\end{figure}

\begin{prop}\label{prop:homstrength}
Suppose $f,g: G \rightarrow H$ are homomorphisms of digraphs.  Then we have
\[f \bhom g \Rightarrow f \dhom g  \Rightarrow f \lhom g.\]
\noindent
Furthermore these implications are strict.
\end{prop}

\begin{proof}
We have seen that $f \bhom g$ if and only if the there exists a path along bidirected edges in the graph $(H^G)^{{o}}$.  The definitions of $\dhom$ and $\lhom$ similarly involve increasingly larger classes of paths in $(H^G)^{{o}}$, and hence the implications hold.  To see that the implications are in fact strict, we refer to the graphs $G$ and $H$ depicted in Figure \ref{fig:directedhomotopy}. In this case the complex $\Hom(G,H)$ consists of 6 isolated points and hence each graph homomorphism $G \rightarrow H$ sits in its own bihomotopy equivalence class. However the directed edges in the graph $H^G$ leads to other notions of homotopies. For instance, if $f=(0,1)~, g=(3,2)$ and $h=(4,5)$ then $f \lhom h $ but $f$ is not dihomotopic to $h$ and $f\dhom g$ but $f$ is not bihomotopic to $g$. Here we let $G = (a,b)$ and use $f=(x,y)$ to denote the images of $(a,b)$ in a homomorphism $f:G \rightarrow H$.
\end{proof}
 
 \begin{figure}[H]
\begin{subfigure}[]{0.2\textwidth}
\vspace{1.5cm}
    \centering
    
    \begin{tikzpicture}[scale=0.85]
\node (a) at (0,0) {a};
\node (b) at (2,0) {b};
\draw[->-=.5] (a) to [out=30,in=150,looseness=0.5] (b);
\draw[->-=.5] (b) to [out=210,in=330,looseness=0.5] (a);
\end{tikzpicture}
\vspace*{1.15cm}
\caption{The graph $G$}
\end{subfigure}
\begin{subfigure}[]{0.2\textwidth}
    \centering
    \begin{tikzpicture}[scale=0.85]
\node (a) at (0,0) {2};
\node (b) at (0,1.5) {0};
\node (c) at (2,1.5) {1};
\node (d) at (2,0) {3};
\node (e) at (0,-1.5) {4};
\node (f) at (2,-1.5) {5};

\draw[->- =.5] (b)--(a) node[midway,above right] {};
\draw[->- =.5] (c)--(d) node[midway,above right] {};
\draw[->-=.5] (b) to [out=30,in=150,looseness=0.3] (c);
\draw[->-=.5] (c) to [out=210,in=330,looseness=0.5] (b);
\draw[->-=.5] (a) to [out=30,in=150,looseness=0.3] (d);
\draw[->-=.5] (d) to [out=210,in=330,looseness=0.5] (a);
\draw[->- =.5] (e)--(a) node[midway,above right] {};
\draw[->- =.5] (f)--(d) node[midway,above right] {};
\draw[->-=.5] (e) to [out=30,in=150,looseness=0.3] (f);
\draw[->-=.5] (f) to [out=210,in=330,looseness=0.5] (e);
\end{tikzpicture}
\vspace{0.1cm}
\caption{The graph $H$}
\end{subfigure}
\begin{subfigure}[]{0.25\textwidth}
    \centering
    \vspace{0.5cm}
    \begin{tikzpicture}[scale=0.8]
\node (d) at (0,1.75) {{\tiny (0,1)}};
\node (e) at (1.5,.25) {{\tiny (2,3)}};
\node (f) at (0,.25) {{\tiny (1,0)}};
\node (g) at (1.5,1.75) {{\tiny (3,2)}};
\node (g) at (3,1.75) {{\tiny (4,5)}};
\node (g) at (3,0.25) {{\tiny (5,4)}};
\node[fill,circle,inner sep=1.5pt,minimum size=1pt] (b) at (1.5,0) {};
\node[fill,circle,inner sep=1.5pt,minimum size=1pt] (c) at (0,0) {};
\node[fill,circle,inner sep=1.5pt,minimum size=1pt] (a) at (0,1.5) {};
\node[fill,circle,inner sep=1.5pt,minimum size=1pt] (i) at (1.5,1.5) {};
\node[fill,circle,inner sep=1.5pt,minimum size=1pt] (j) at (3,1.5) {};
\node[fill,circle,inner sep=1.5pt,minimum size=1pt] (k) at (3,0) {};
\end{tikzpicture}
\vspace{1cm}
\caption{$\Hom(G,H)$}
\end{subfigure}
\begin{subfigure}[]{0.33\textwidth}
    \centering
    \begin{tikzpicture}[scale=0.85]
\node (a) at (-2,2) {{\tiny (0,1)}};
\node (b) at (0,0) {{\tiny (2,3)}};
\node (c) at (-2,0) {{\tiny (1,0)}};
\node (d) at (0,2) {{\tiny (3,2)}};
\node (e) at (2,2) {{\tiny (4,5)}};
\node (f) at (2,0) {{\tiny (5,4)}};
\draw[->- =.5] (a)--(d) node[midway,above right] {};
\draw[->- =.5] (c)--(b) node[midway,above right] {};
\draw[->- =.5] (e)--(d) node[midway,above right] {};
\draw[->- =.5] (f)--(b) node[midway,above right] {};
\draw[->-=.5] (a) to [out=130,in=60,looseness=8] (a);
\draw[->-=.5] (b) to [out=130,in=60,looseness=8] (b);
\draw[->-=.5] (c) to [out=130,in=60,looseness=8] (c);
\draw[->-=.5] (d) to [out=130,in=60,looseness=8] (d);
\draw[->-=.5] (e) to [out=130,in=60,looseness=8] (e);
\draw[->-=.5] (f) to [out=130,in=60,looseness=8] (f);
\end{tikzpicture}
\vspace{0.1cm}
\caption{The graph $(H^G)^{{o}}$}
\end{subfigure}
   \caption{An illustration of the implications of our various homotopies.}\label{fig:directedhomotopy}
\end{figure}
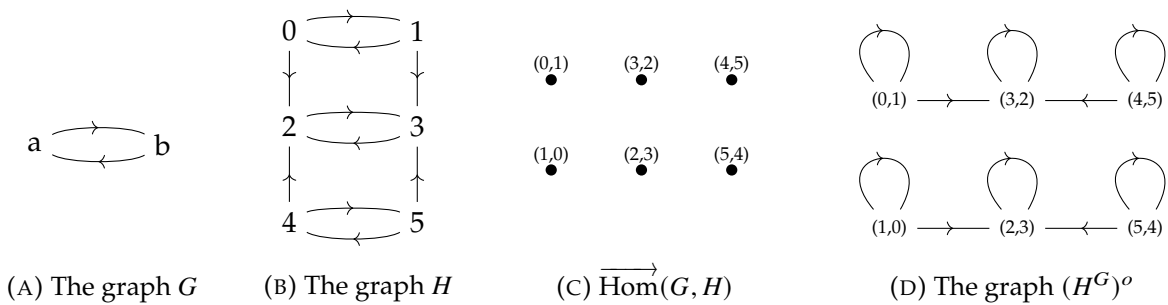

\subsection{Higher homotopies}\label{sec:higherhom}
We have seen that bihomotopies of digraph homomorphisms correspond to paths in the one skeleton of the complex $\Hom(G,H)$.  Hence we can take advantage of the topology of $\Hom(G,H)$ to define higher categorical structures.  In particular, our results from Section \ref{sec:structural} implies that the category of directed graphs is \emph{enriched} over the category of posets (and hence topological spaces). For the case of undirected graphs, this is further studied in \cite{Doc} and \cite{DocSch}.

It would be interesting to construct a topological space that extends paths for our other notions of of homotopy. In the case of line homotopy, a natural candidate is provided by the \defi{directed clique complex} of a directed graph.  This notion has recently found applications in network theory, see for instance \cite{GPCI} and \cite{MasVil}.  To recall the definition suppose $G$ is a digraph. Then the \defi{directed clique complex} $\overrightarrow{X}(G)$ is the simplicial complex on vertex set $V(G)$ with $(n-1)$-simplices given by all \emph{directed} $n$-cliques:
\[\{(v_1, \dots, v_n):(v_i,v_j) \in E(G) \text{ for all $1 \leq i < j \leq n$}\}.\]

\noindent
Note that a directed $n$-clique is simply a subgraph of $G$ that is isomorphic to $\tourn$.  Note that the edges in $\overrightarrow{X}((H^G)^{{o}})$ are precisely the (directed) edges in $(H^G)^{{o}}$ so a path in $\overrightarrow{X}((H^G)^{{o}})$ corresponds precisely to a (not necessarily) directed path in $(H^G)^{{o}}$. 

The notion of line-homotopy equivalence between homomorphisms of digraphs naturally leads to a definition of line-homotopy equivalence of digraphs $G \lhom H$, and it is an open question how this relates to the topology of the directed clique complex.
For instance, if $G \lhom H$ are line homotopy equivalent digraphs, is it true that $\overrightarrow{X}((G^T)^{{o}})$ and $\overrightarrow{X}((H^T)^{{o}})$ are homotopy equivalent for any digraph $T$? More generally, do the results from Proposition \ref{prop:equiv} hold in this setting? If so this would lead to a homology theory for digraphs that is invariant under line homotopy equivalence.

\begin{rem}
The $\times$-homotopy theory for undirected graphs developed in \cite{Doc} has close connections to the notion of \emph{strong homotopy type} as introduced by Barmak and Minian in \cite{BarMin}. Here one can see that the strong collapse of a simplicial complex from \cite{BarMin}, when restricted to a clique complex $X(G)$ of a graph $G$, correspond to \emph{foldings} of $G$.  In \cite{Doc} it is shown that foldings of a graph $G$ in generate the $\times$-homotopy type of a graph and in particular preserve the homotopy type of $\oldHom(-,G)$. It would be interesting to see how the notion of directed homotopy and directed foldings can be generalized to a setting of directed simplicial complexes.
\end{rem}

\section{Further thoughts}\label{sec:further}

In this section we discuss some areas of future research and collect some open questions.

\subsection{Topology of complexes into tournaments}

 In Section \ref{sec:tourn} we saw that if $\tourn$ is the transitive tournament on $n$ vertices then for any digraph $G$ the complex $\Hom(G,\tourn)$ is empty or contractible. 
On the other hand we saw in Lemma \ref{prop:spheres} that for any $n$ there exists a tournament $T_{2n+3}$ on $2n+3$ vertices such that $\N(T_{2n+3})$ is homotopy equivalent to an $n$-sphere ${\mathbb S}^n$. What are the other homotopy types that can be achieved?  For instance, we do not know of any tournament $T_n$ for which the neighborhood complex $\N(T_n)$ has torsion in its homology. In a related question, does the topology of $\N(T_n)$ say something about the combinatorics of the tournament?

In the context of undirected graphs, the maximum degree of a graph $G$ influenced the topology of $\oldHom(G,K_n)$.  In particular, \v{C}uk\'ic and Kozlov \cite{CukKozHig} proved that if $G$ has max degree $d$ then $\oldHom(G,K_n)$ is at least $(n-d-2)$-connected.  A stronger result in terms of \emph{degeneracy} of $G$ was established by Malen in \cite{Mal}. We note that something similar cannot hold for digraphs, at least for arbitrary tournaments.  In particular, it is not hard to see that for any $n$ there exists a tournament $T_n$ for which $\N(T_n) \simeq \Hom(K_2,T_n)$ is disconnected.

\subsection{Obstructions to graph homomorphisms}

In the context of undirected graphs, the topology of the neighborhood (and more general homomorphism) complexes of a graph $G$ leads to lower bounds on the chromatic number $\chi(G)$. In particular, by results of Lov\'asz \cite{Lov} and Babson-Kozlov \cite{BabKozPro}, we have for any graph $G$ the following bounds:
\begin{equation}\label{eq:bounds}
\begin{split}
 \chi(G) &\geq \conn({\mathcal N}(G)) + 3 = \conn(\oldHom(K_2,G)) +3;\\
\chi(G) &\geq \conn(\oldHom(C_{2r+1},G)) + 4.
\end{split}
\end{equation}

Here $C_{2r+1}$ denotes an odd cyle on $2r+1$ vertices, and for a topological space $X$, $\conn(X)$ is the \emph{connectivity} of $X$. These results follow from an obstruction theory for graph homomorphisms that utilizes the ${\mathbb Z}_2$-equivariant topology of the relevant Hom complexes.  In this context the free group action on the topological spaces is induced by a ${\mathbb Z}_2$-action on the graphs $K_2$ and $C_{2r+1}$. By taking $G$ to be $K_2$ (resp. $C_{2r+1}$) one can see that the bounds in Equation \ref{eq:bounds} are tight. In the language of Hom complexes we say that $K_2$ and $C_{2r+1}$ are \defi{test graphs} (see \cite{DocSch} and \cite{KozChrom}).

It is a natural question to ask whether there is a similar obstruction theory coming from the homomorphism complexes of digraphs.  In this context there is no nontrivial action on the directed edge $K_2$ and new tools seem to be required. However, if we view the edge $K_2$ as a directed $2$-cycle $C_2$, the next natural nontrivial group action arises if we consider the oriented cycle $\overrightarrow{C}_3$ as a potential test graph.  We then have a free ${\mathbb Z}_3$ group action on $\overrightarrow{C}_3$ which induces a free ${\mathbb Z}_3$-action on the complex $\Hom(\overrightarrow{C}_3, G)$ for the case of loop-free $G$. Since $\Hom(\overrightarrow{C}_3,-)$ is functorial, we can then employ Dold's Theorem (see for instance \cite{Mat}).

\begin{thm}[Dold]
Let $\Gamma$ be a non-trivial finite group. Suppose $X$ and $Y$ are $\Gamma$-complexes where the action of $\Gamma$ on $Y$ is free and $\dim Y \leq n$. Then if $X$ is $n$-connected there cannot exist a continuous $\Gamma$-equivariant map $X \rightarrow Y$.
\end{thm}

Knowledge of the dimension of $\Hom(\overrightarrow{C}_3,H)$ for any digraph $H$ then leads to an obstruction theory for homomorphisms $G \rightarrow H$. For example, in Figure \ref{fig:example2} we have a tournament $T_5$ with the property that $\Hom(\overrightarrow{C}_3,T_5)$ is 1-dimensional.  Hence if $G$ is any digraph with simply-connected $\Hom(\overrightarrow{C}_3,G)$ then there cannot exist a homomorphism $G \rightarrow T_5$. Similarly for $T_7$ in Figure \ref{fig:example4} we have that if $\Hom(\overrightarrow{C}_3,G)$ is $2$-connected there cannot exist a homomorphism $G \rightarrow T_7$.

\subsection{Other directions}

There are a number of other properties and applications of $\Hom$ complexes that one could explore.  Some of these have been mentioned in previous sections but we collect them here.

\begin{itemize}

    \item \emph{Other graph classes} - In addition to tournaments it would be interesting to investigate the neighborhood complexes of other classes of digraphs. Do the combinatorial properties of these digraphs play a role in the topology (and diameter) of the resulting complexes?   Natural candidates include Eulerian digraphs, digraphs of bounded width, etc.
    
    \smallskip
    \item \emph{Directed cops and robbers} - In the undirected setting one defines a pursuit-evasion game on the vertices of a graph $G$, and it is known that a finite $G$ is \emph{cop-win} if and only if $G$ is dismantlable \cite{NowWin}.  We do not know if there is a similar characterization (or even a similar game) for digraphs.
    
    \smallskip
    \item \emph{Gibbs measures and long range action} - In \cite{BriWinGibbs} Brightwell and Winkler study a notion of \emph{long range action} as well as Gibbs measures on the set of graph homomorphisms between undirected graphs.  They again relate these properties to dismantlable graphs and chromatic number.  We do not know if similar constructions apply in the directed setting.
    
    \smallskip
\item \emph{Homology theories} - As we have discussed, the homotopy theory of digraphs introduced in \cite{GLMY} was developed in tandem with an underlying homology theory.  These constructions can be seen as digraphs analogues of $A$-theory for undirected graphs studied in \cite{BBDL}.  In all of these cases, the underlying product is the cartesian product of (di)graphs.  Our three notions of homotopy described in Section \ref{sec:homotopy} are based on the categorical product and lead to new constructions. For instance, for any two digraphs $G$ and $H$ we can define the $G$-bihomology groups of $H$ to be $\tilde H_i(\Hom(G,H))$.  Our results from Section \ref{sec:homotopy} then imply that $G$-homology is invariant under bihomotopy equivalence of digraphs.  Similarly, we can define the $G$-line homology groups of $H$ to be $\tilde{H}_i( \overrightarrow{X}((H^G)^{{o}}))$ which conjecturally (see Section \ref{sec:higherhom}) is invariant under line homotopy equivalence of digraphs.  Similar constructions were recently considered by Bubenik and Mili\'cevi\'c in \cite{BubMil} although we are not aware of the precise connection.
    
    \smallskip
    \item \emph{Model categories} - In any category with a notion of equivalence, a natural question to ask if one can impose a model category structure.  In the context of undirected graphs, Matsushita \cite{MatBox} described a model category structure using ${\mathbb Z}_2$-homotopy equivalence of neighborhood complexes as weak equivalences.   In \cite{Dro} Droz constructed model categories with other notions of weak equivalences. On the other hand, in \cite{GoySan} Goyal and Santhanam showed no model category structure exists if one takes $\times$-homotopy as weak equivalence and inclusions as cofibrations.  It is an open question whether one can use any of the three notions of homotopy equivalence for digraphs discussed here in the context of a model category structure.  Model categories have been studied in the context of directed homotopy for instance in work of Gaucher \cite{Gau}.
    
    \smallskip
    \item \emph{Stanley-Reisner rings} - For any simplicial complex $X$ on vertex set $[n]$, one defines the \emph{Stanley-Reisner} ideal $I(X) \subset R = {\mathbb K}[x_1, \dots, x_n]$ as the monomial ideal in $R$ generated by nonfaces of $X$. In Corollary \ref{cor:Leray} we saw that if $G$ is a simple digraph on at most $m = 2n+2$ vertices then the neighborhood complex $\N(G)$ is $n$-Leray.  This has consequences for the ideal $I(
    \N(G))$, and in particular the regularity of the factor ring $R/I(\N(G))$.  It would be interesting to study other algebraic invariants of these rings, and to relate them to combinatorial properties of the underlying digraph $G$.
    
\end{itemize}

\end{document}